\documentclass[12pt]{article}

\usepackage[T1]{fontenc}
\usepackage{lmodern}
\usepackage{amsfonts,amssymb,amsmath,amsthm,mathtools,mathrsfs}
\usepackage{hyperref}
\usepackage{color}
\newtheorem{theorem}{Theorem}[section]
\newtheorem{proposition}[theorem]{Proposition}
\newtheorem{lemma}[theorem]{Lemma}
\newtheorem{corollary}[theorem]{Corollary}

\theoremstyle{definition}
\newtheorem{definition}[theorem]{Definition}

\theoremstyle{remark}
\newtheorem{remark}[theorem]{Remark}

\newcommand{\N}{\mathbb N}
\newcommand{\R}{\mathbb R}
\newcommand{\Qp}{\mathbb Q_p}
\newcommand{\Norm}{\mathscr N}
\newcommand{\dinf}{d_{\infty}}
\newcommand{\Ocoeff}{\mathcal O_E}

\newcommand{\Building}{\mathcal B}
\newcommand{\cInd}{\operatorname{c\text{-}Ind}}

\DeclareMathOperator{\GL}{GL}
\DeclareMathOperator{\Hom}{Hom}
\DeclareMathOperator{\End}{End}

\title{Space of norms on locally algebraic representations}
\author{Alexandre Pyvovarov}
\date{\today}

\begin{document}
\maketitle

\begin{abstract}
Let $F$ and $E$ be finite extensions of $\Qp$, let $\mathbb G$ be a
reductive group over $F$, and put $G=\mathbb G(F)$.  Let $V$ be a
locally algebraic representation of the form
$V=\pi_{\mathrm{sm}}\otimes_E\sigma_{\mathrm{alg}}$, where
$\pi_{\mathrm{sm}}$ is smooth admissible and
$\sigma_{\mathrm{alg}}$ is finite-dimensional algebraic.  We study the
extended Goldman--Iwahori distance on the set of non-Archimedean norms
on $V$.  After fixing a reference norm $\alpha_0$, its finite-distance
component $\Norm_{\alpha_0}(V)$ is the bounded projective limit of the
extended Bruhat--Tits buildings attached to
$V_K=\pi_{\mathrm{sm}}^K\otimes_E\sigma_{\mathrm{alg}}$.  It is complete
for the resulting uniform sup metric; this metric is of $\ell^\infty$
type and is generally not CAT(0).  We prove directly that a $G$-orbit in
$\Norm_{\alpha_0}(V)$ is bounded if and only if this component contains a
$G$-invariant norm.  The invariant norm is the pointwise supremum of the
orbit.  We formulate an integral group-algebra and type-Hecke condition
necessary for an invariant norm.  For $G=\GL_n(F)$ we specialise to
$V=\operatorname{BS}(r)=\pi_{\mathrm{gen}}(r)\otimes_E
\pi_{\mathrm{alg}}(r)$.
\end{abstract}

\tableofcontents

\section{Introduction}

Let $F$ and $E$ be finite extensions of $\Qp$, with fixed absolute value
$|\cdot|$ on $E$.  Let $\mathbb G$ be a reductive group over $F$ and let
$G=\mathbb G(F)$.  Fix a smooth admissible $E$-representation
$\pi_{\mathrm{sm}}$ of $G$ and a finite-dimensional algebraic
$E$-representation $\sigma_{\mathrm{alg}}$ of
$\operatorname{Res}_{F/\Qp}\mathbb G$, and put
\begin{equation}\label{eq:locally-algebraic-factorisation}
  V=\pi_{\mathrm{sm}}\otimes_E\sigma_{\mathrm{alg}}
\end{equation}
with the diagonal action.  Thus $V$ is locally algebraic.  We regard
the factorisation \eqref{eq:locally-algebraic-factorisation} as part of
the data.  Finite direct sums of such representations are handled by
the same argument.  In the arithmetic discussion one may additionally
assume that $\pi_{\mathrm{sm}}$ is finitely generated and has countable
$E$-dimension; neither hypothesis is needed for the metric statements.

The purpose of this paper is to separate a general geometric fact from
the difficult arithmetic question that motivated it.  The geometric
fact is the following theorem.

\begin{theorem}\label{thm:intro}
Fix a norm $\alpha_0$ on $V$.  The finite-distance component
$\Norm_{\alpha_0}(V)$ contains a $G$-invariant norm if and only if the
orbit $G\alpha_0$ is bounded for the Goldman--Iwahori metric.
\end{theorem}

The proof is elementary.  If the orbit is bounded, then
\[
  \alpha_G(v)=\sup_{g\in G}\alpha_0(g^{-1}v)
\]
is a finite, nondegenerate, $G$-invariant norm equivalent to
$\alpha_0$.  Thus the fixed-point step does not require a CAT(0)
circumcentre argument.

The metric used here is the sup metric $\dinf$.  On a diagonal apartment
it is an $\ell^\infty$ metric, rather than the Euclidean metric of a
Bruhat--Tits building.  Goldman and Iwahori introduced this metric in
finite dimension \cite{MR144889}; its injective geometry has been
developed further by Haettel \cite{Haettel2022}.  It is important not to
identify $\dinf$ with the usual CAT(0) metric.

The geometry can be pushed further.  We prove below that every finite-distance
component is hyperconvex, hence injective, by an explicit supremum
construction.  Every bounded family of norms has circumradius exactly half
its diameter, with an explicit circumcentre.  For a bounded \(G\)-orbit this
gives the sharp identity
\[
 \inf_{\beta\in\Norm_{\alpha_0}(V)^G}
 \dinf(\alpha_0,\beta)
 =
 \frac12
 \operatorname{diam}(G\alpha_0).
\]
These facts do not manufacture a bounded orbit, but they remove all loss in
the metric fixed-point step once arithmetic has supplied boundedness.

The arithmetic issue is therefore precise: one must construct a norm
whose entire $G$-orbit is uniformly bounded.  Inequalities for Jacquet
exponents or Hecke eigenvalues are necessary, but turning them into one
uniform estimate is the substantive step.  This is closely related to
the integrality questions in the Breuil--Schneider conjecture
\cite{MR2359853,MR3409331}.  Known proofs in significant cases use
global methods, patching, or explicit integral Hecke modules
\cite{Sorensen2013,CEGGPS2016,Pyvovarov2021}.

There is also a useful exact finite-dimensional approximation.  The
decomposition
\[
  V=\bigcup_K\bigl(\pi_{\mathrm{sm}}^K\otimes_E\sigma_{\mathrm{alg}}\bigr)
\]
identifies the set of all norms with a projective limit of extended
Bruhat--Tits buildings.  The finite-distance component of a reference
norm is the subset on which the distances from the induced basepoints
are bounded uniformly in $K$.  This ``bounded projective limit'' is the
correct object for the metric problem.  It is stronger than the ordinary
coordinatewise projective limit, and the distinction is exactly the
uniformity needed for bounded orbits.

For a matrix $M$ with coefficients in $E$, define $v_E(M)=\min_{a,b}v_E(M_{ab})$, where $v_E$ is the normalized valuation on $E$.

\section{The locally algebraic setting}
\label{sec:locally-algebraic}

\subsection{The finite-dimensional exhaustion}

Let $K\subseteq G$ be compact open.  Define
\begin{equation}\label{eq:smooth-union}
 V_K=\pi_{\mathrm{sm}}^K\otimes_E\sigma_{\mathrm{alg}}\subseteq V.
\end{equation}
If $L\subseteq K$, then $V_K\subseteq V_L$.  Smoothness and admissibility
give
\begin{equation}\label{eq:locally-algebraic-union}
 V=\bigcup_KV_K,\qquad
 \dim_EV_K=(\dim_E\pi_{\mathrm{sm}}^K)(\dim_E\sigma_{\mathrm{alg}})<\infty.
\end{equation}
These spaces are not, in general, the fixed spaces $V^K$: the compact
subgroup acts nontrivially on the algebraic factor.

\begin{lemma}[Reconstruction from algebraic levels]\label{lem:reconstruct}
Suppose that for every compact open subgroup $K$ a norm $\alpha_K$ is
given on $V_K$, and whenever $L\subseteq K$ one has
$\alpha_L|_{V_K}=\alpha_K$.  Then there is a unique norm $\alpha$ on $V$
whose restriction to $V_K$ is $\alpha_K$ for every $K$.
\end{lemma}

\begin{proof}
For $v\in V$, choose $K$ with $v\in V_K$ and set
$\alpha(v)=\alpha_K(v)$.  If $v\in V_K\cap V_L$, choose a compact open
subgroup $M\subseteq K\cap L$; compatibility gives the same value on
$V_M$.  Homogeneity and the ultrametric inequality reduce to a common
finite-dimensional level.
\end{proof}

\subsection{Norms and the extended distance}

\begin{definition}
A \emph{seminorm} on $V$ is a function
$\alpha:V\rightarrow\R_{\geq 0}$ satisfying
\begin{align*}
  \alpha(\lambda v)&=|\lambda|\alpha(v),\\
  \alpha(v+w)&\leq\max\{\alpha(v),\alpha(w)\}
\end{align*}
for $\lambda\in E$ and $v,w\in V$.  It is a \emph{norm} if in addition
$\alpha(v)=0$ implies $v=0$.
\end{definition}

The set of norms on any $E$-vector space is nonempty.  Indeed, choose a
Hamel basis $(e_i)_{i\in I}$ and positive weights $(c_i)_{i\in I}$, and
put
\[
  \alpha\left(\sum_{i\in I}a_i e_i\right)
     =\max_i |a_i|c_i,
\]
where only finitely many $a_i$ are nonzero.

For two norms $\alpha$ and $\beta$ define
\begin{equation}\label{eq:distance}
  \dinf(\alpha,\beta)
   =\sup_{v\in V\setminus\{0\}}
     \left|\log\frac{\alpha(v)}{\beta(v)}\right|
   \in\R_{\geq0}\cup\{+\infty\}.
\end{equation}
This is an extended metric.  For $R\geq0$, the inequality
$\dinf(\alpha,\beta)\leq R$ is equivalent to
\begin{equation}\label{eq:comparison}
  e^{-R}\beta(v)\leq\alpha(v)\leq e^R\beta(v)
  \qquad(v\in V).
\end{equation}

For a compact open subgroup $K$, write $d_K$ for the same distance on
the finite-dimensional space $V_K$.  The exhaustion
\eqref{eq:locally-algebraic-union} gives
\begin{equation}\label{eq:distance-fixed}
  \dinf(\alpha,\beta)
    =\sup_Kd_K(\alpha|_{V_K},\beta|_{V_K}).
\end{equation}
Each term on the right is finite, but their supremum need not be.

\begin{definition}\label{def:component}
For a reference norm $\alpha_0$, its \emph{finite-distance component}
is
\[
  \Norm_{\alpha_0}(V)
    =\{\alpha:\dinf(\alpha,\alpha_0)<\infty\}.
\]
The restriction of $\dinf$ to this set is a genuine metric.
\end{definition}

Different choices of reference norm give either the same component or
disjoint components.  Thus there is generally no single metric space of
all norms: there is a disjoint union of finite-distance components.

\section{Bounded projective limits}
\label{sec:inverse-limit}

\subsection{The projective system}

Let $\mathcal K(G)$ be the set of compact open subgroups of $G$, ordered
by reverse inclusion, and put
\[
 \Building_K^{\mathrm e}=\Building^{\mathrm e}(\GL(V_K),E).
\]
The points of $\Building_K^{\mathrm e}$ may be realised as norms on
$V_K$ \cite{MR144889,MR2439729}.  If $L\subseteq K$, then
$V_K\subseteq V_L$, and restriction gives a surjective, one-Lipschitz
map
\[
  r_{L,K}:\Building_L^{\mathrm e}\longrightarrow
  \Building_K^{\mathrm e},\qquad
  x\longmapsto x|_{V_K}.
\]
The maps satisfy $r_{K,K}=1$ and
$r_{L,K}r_{M,L}=r_{M,K}$ for $M\subseteq L\subseteq K$.  Thus they
form a projective system.

\begin{proposition}[Norms as a projective limit]
\label{prop:all-norms-limit}
Restriction induces a canonical bijection
\[
  \{\text{norms on }V\}
   \xrightarrow{\ \sim\ }
  \varprojlim_{K\in\mathcal K(G)}\Building_K^{\mathrm e}.
\]
Under this bijection,
\[
  \dinf(\alpha,\beta)
   =\sup_K d_K(\alpha|_{V_K},\beta|_{V_K}),
\]
where both sides are allowed to be infinite.
\end{proposition}

\begin{proof}
Compatibility of the restrictions is clear.  Conversely, a compatible
family $(x_K)_K$ gives a unique norm on $V$ by Lemma
\ref{lem:reconstruct}.  The metric identity is
\eqref{eq:distance-fixed}.
\end{proof}

Write $o_K=\alpha_0|_{V_K}$.  The preceding proposition gives the more
precise description
\begin{equation}\label{eq:bounded-projective-limit}
 \Norm_{\alpha_0}(V)
 \cong
 \left\{(x_K)_K\in
   \varprojlim_K\Building_K^{\mathrm e}:
   \sup_K d_K(x_K,o_K)<\infty\right\}.
\end{equation}
This is the \emph{bounded projective limit} based at
$o=(o_K)_K$, and its metric is
\begin{equation}\label{eq:inverse-limit-metric}
 d_{\mathrm{bd}}(x,y)=\sup_Kd_K(x_K,y_K).
\end{equation}
In particular, for every $R\geq0$,
\begin{equation}\label{eq:fixed-radius-limit}
 \overline B_{\Norm}(\alpha_0,R)
 \cong
 \varprojlim_K\overline B_{\Building_K^{\mathrm e}}(o_K,R),
 \qquad
 \Norm_{\alpha_0}(V)
 =\bigcup_{m\in\N}\overline B_{\Norm}(\alpha_0,m).
\end{equation}
The ordinary projective-limit topology is only the topology of
coordinatewise convergence.  The topology from
\eqref{eq:inverse-limit-metric} is the stronger topology of uniform
convergence over all $K$.  Confusing the two loses precisely the
uniform estimate needed for a bounded orbit.

\subsection{\texorpdfstring{The $G$-action on the system}{The G-action on the system}}

For $g\in G$ and $K\in\mathcal K(G)$ define an isometry
\[
 \tau_{g,K}:
 \Building_{g^{-1}Kg}^{\mathrm e}\longrightarrow
 \Building_K^{\mathrm e},\qquad
 (\tau_{g,K}x)(v)=x(g^{-1}v)\quad(v\in V_K).
\]
Indeed, the diagonal action gives
$gV_{g^{-1}Kg}=V_K$: on the smooth factor it sends
$\pi_{\mathrm{sm}}^{g^{-1}Kg}$ to $\pi_{\mathrm{sm}}^K$, and on the
algebraic factor it is an automorphism of
$\sigma_{\mathrm{alg}}$.
These transport maps commute with restriction and satisfy the evident
cocycle identity.  If $x=(x_K)_K$ represents a global norm, then
\begin{equation}\label{eq:action-on-limit}
  (gx)_K=\tau_{g,K}(x_{g^{-1}Kg}).
\end{equation}
Consequently,
\begin{equation}\label{eq:orbit-on-levels}
 \dinf(\alpha,g\alpha)
 =\sup_Kd_K\!\left(
   \alpha|_{V_K},
   \tau_{g,K}(\alpha|_{V_{g^{-1}Kg}})
 \right).
\end{equation}
Formula \eqref{eq:orbit-on-levels} makes clear that boundedness on every
individual $V_K$ is insufficient: the bound must be independent of
both $K$ and $g$.

\subsection{Compactness and finite diagrams}

Closed bounded subsets of each finite-dimensional building
$\Building_K^{\mathrm e}$ are compact.  This gives a useful
local-to-global principle.  A \emph{finite building diagram} means
finitely many equations, each of one of the forms
\[
 r_{L,K}(x_L)=x_K\quad(L\subseteq K),\qquad
 \tau_{g,K}(x_{g^{-1}Kg})=x_K\quad(g\in G).
\]

\begin{theorem}[Uniform finite-diagram criterion]
\label{thm:finite-diagrams}
Fix $R\geq0$.  There is a $G$-invariant
$\beta\in\Norm_{\alpha_0}(V)$ with
$\dinf(\beta,\alpha_0)\leq R$ if and only if every finite building
diagram has a solution with
\[
  x_K\in\overline B_{\Building_K^{\mathrm e}}(o_K,R)
\]
at every vertex occurring in the diagram.
\end{theorem}

\begin{proof}
Necessity follows by restricting $\beta$.  For sufficiency, consider
the compact product
\[
  X_R=\prod_K\overline B_{\Building_K^{\mathrm e}}(o_K,R).
\]
Every restriction or transport equation cuts out a closed subset of
$X_R$.  The hypothesis is exactly the finite-intersection property for
this collection of closed subsets.  Compactness of $X_R$ gives a point
satisfying all equations.  The restriction equations reconstruct a
global norm by Proposition \ref{prop:all-norms-limit}; the transport
equations and \eqref{eq:action-on-limit} make it $G$-invariant.
\end{proof}

\begin{corollary}[Inverse-limit test for a bounded orbit]
\label{cor:inverse-limit-bounded}
The orbit $G\alpha_0$ is bounded if and only if there is one
$R<\infty$ for which all finite building diagrams in Theorem
\ref{thm:finite-diagrams} are solvable.  If
\[
 D=\sup_{g\in G}\dinf(\alpha_0,g\alpha_0)<\infty,
\]
one may take $R=D$.
\end{corollary}

\begin{proof}
The first assertion follows from Theorems \ref{thm:bounded-orbit} and
\ref{thm:finite-diagrams}.  For the last assertion, the invariant
supremum norm \eqref{eq:invariant-sup} satisfies
$\alpha_0\leq\alpha_G\leq e^D\alpha_0$.
\end{proof}

This compactness statement is the useful leverage supplied by the
inverse limit.  It does not manufacture the uniform radius $R$:
arithmetic input must do that.  Once a common $R$ is known, however,
one may solve larger and larger finite Hecke/type diagrams without
choosing compatible solutions by hand; compactness extracts a single
compatible and equivariant family.

\subsection{Integral form}

Let $\varpi_E$ be a uniformiser of $E$, let
$q_E$ be the cardinality of its residue field, and normalise
$|\varpi_E|=q_E^{-1}$.  The unit ball of any norm gives a gauge norm at
distance at most $\log q_E$ from it.  Thus, at the price of this fixed
rounding error, $\alpha_0$ may be assumed to be the gauge of a full
$\Ocoeff$-lattice $\Lambda_0\subset V$.  Put
$\Lambda_{0,K}=\Lambda_0\cap V_K$.

\begin{proposition}[Compatible lattice criterion]
\label{prop:lattice-limit}
There is a $G$-stable full lattice $\Lambda\subset V$ commensurable
with $\Lambda_0$ if and only if there are an integer $N\geq0$ and full
lattices $\Lambda_K\subset V_K$ such that, for all $L\subseteq K$ and
$g\in G$,
\begin{align}
 \Lambda_L\cap V_K&=\Lambda_K,\label{eq:lattice-restriction}\\
 g\Lambda_{g^{-1}Kg}&=\Lambda_K,\label{eq:lattice-transport}\\
 \varpi_E^N\Lambda_{0,K}
  \subseteq\Lambda_K
  &\subseteq\varpi_E^{-N}\Lambda_{0,K}.
  \label{eq:lattice-uniform}
\end{align}
For a fixed $N$, it is enough that every finite collection of
\eqref{eq:lattice-restriction}--\eqref{eq:lattice-transport} be
solvable subject to \eqref{eq:lattice-uniform}.
\end{proposition}

\begin{proof}
A global lattice gives the displayed family by intersection.  In the
other direction, declare that $v\in\Lambda$ if
$v\in\Lambda_K$ for one (equivalently every) $K$ with $v\in V_K$.
Equation \eqref{eq:lattice-restriction} makes this well defined and
shows that $\Lambda$ is an $\Ocoeff$-module; the uniform inclusions
make it a full lattice commensurable with $\Lambda_0$.
Equation \eqref{eq:lattice-transport} makes it $G$-stable.  For the
last assertion, note that the set of lattices between the two bounds
in \eqref{eq:lattice-uniform} is finite at each $K$, because the
residue field is finite.  The same finite-intersection argument as in
Theorem \ref{thm:finite-diagrams} applies.
\end{proof}

The gauge of $\Lambda$ in Proposition \ref{prop:lattice-limit} is a
$G$-invariant norm.  The proposition also identifies the exact gap in
a Hecke-only argument: a lattice stable under a double-coset average
must be promoted to lattices satisfying restriction, individual
transport, and one uniform commensurability exponent $N$.

\subsection{Transport operators in the bounded projective limit}
\label{subsec:inverse-limit-transport}

The projective-limit description is also useful for estimating individual
group elements.  Let
\[
 \alpha=(\alpha_K)_K\in\Norm_{\alpha_0}(V).
\]
For \(g\in G\) and a compact open subgroup \(K\), the map \(g\) induces a
linear isomorphism
\[
 g:V_{g^{-1}Kg}\longrightarrow V_K.
\]
Define its finite-level transport norm by
\begin{equation}\label{eq:finite-level-transport-norm}
 \|g\|_{\alpha,K}
 =
 \sup_{0\neq v\in V_{g^{-1}Kg}}
 \frac{\alpha_K(gv)}
 {\alpha_{g^{-1}Kg}(v)}.
\end{equation}

\begin{proposition}[Exact inverse-limit transport formula]
\label{prop:inverse-limit-transport}
For every \(g\in G\),
\begin{equation}\label{eq:global-operator-as-sup-levels}
 \|g\|_\alpha
 :=
 \sup_{0\neq v\in V}\frac{\alpha(gv)}{\alpha(v)}
 =
 \sup_K\|g\|_{\alpha,K}.
\end{equation}
Consequently,
\begin{equation}\label{eq:dinf-inverse-limit-transport}
 \dinf(\alpha,g\alpha)
 =
 \sup_K
 \max\left\{
 \log\|g\|_{\alpha,K},
 \log\|g^{-1}\|_{\alpha,K}
 \right\}.
\end{equation}
More generally,
\begin{equation}\label{eq:dinf-inverse-limit-powers}
 \dinf(\alpha,g^m\alpha)
 =
 \sup_K
 \max\left\{
 \log\|g^m\|_{\alpha,K},
 \log\|g^{-m}\|_{\alpha,K}
 \right\}.
\end{equation}
\end{proposition}

\begin{proof}
If \(v\in V_L\), take \(K=gLg^{-1}\).  Then
\(v\in V_{g^{-1}Kg}\), so every vector occurring in the global operator norm
occurs in one of the finite-level norms.  The reverse inequality is
immediate because the finite-level norms are restrictions of \(\alpha\).
This proves \eqref{eq:global-operator-as-sup-levels}.

For the action convention
\[
 (g\alpha)(v)=\alpha(g^{-1}v),
\]
the comparison constant between \(\alpha\) and \(g\alpha\) is
\[
 \max\{\|g\|_\alpha,\|g^{-1}\|_\alpha\}.
\]
Combining this with \eqref{eq:global-operator-as-sup-levels} gives
\eqref{eq:dinf-inverse-limit-transport}; the same argument applied to
\(g^m\) gives \eqref{eq:dinf-inverse-limit-powers}.
\end{proof}

\begin{corollary}[Uniform finite-level Cartan criterion]
\label{cor:uniform-finite-level-Cartan}
Let \(S\subseteq G\).  A norm
\(\alpha=(\alpha_K)_K\in\Norm_{\alpha_0}(V)\) has bounded \(S\)-orbit if and
only if
\[
 \sup_{g\in S}\sup_K
 \max\left\{
 \log\|g\|_{\alpha,K},
 \log\|g^{-1}\|_{\alpha,K}
 \right\}
 <\infty.
\]
In particular, for a Cartan generator \(a_j\), uniform boundedness of
\(\{a_j^m\alpha:m\in\mathbb Z\}\) is exactly a bound which is uniform
simultaneously in \(m\) and \(K\).
\end{corollary}

This changes the role of a Hamel basis.  A single global Hamel basis is
convenient for displaying one global matrix, but it is not necessary for
proving boundedness.  It is enough to construct compatible norms or lattices
on the finite-dimensional spaces \(V_K\) and to obtain one transport bound
uniform in \(K\).

\begin{lemma}[Moving-level transport chains]
\label{lem:moving-level-transport-chain}
Fix \(g\in G\), a compact open subgroup \(K\), and a full lattice
\(\Lambda_K\subseteq V_K\).  Put
\[
 K_r=g^rKg^{-r},
 \qquad
 \Lambda_{K_r}=g^r\Lambda_K.
\]
Then
\[
 g:\Lambda_{K_r}\xrightarrow{\ \sim\ }\Lambda_{K_{r+1}}
\]
for every \(r\in\mathbb Z\).  Thus every finite chain of conjugate levels
has zero transport loss when the lattices are allowed to move with the
levels.

The obstruction to a global lattice lies instead in compatibility on common
refinements and in one commensurability exponent uniform as the finite chain
grows.
\end{lemma}

\begin{proof}
The first assertion is tautological.  The second is exactly the additional
content of the restriction and uniformity conditions in the bounded
projective-limit lattice criterion.
\end{proof}

\section{Metric geometry}\label{sec:metric}

\begin{proposition}\label{prop:complete}
The metric space $(\Norm_{\alpha_0}(V),\dinf)$ is complete.
\end{proposition}

\begin{proof}
Let $(\alpha_n)$ be Cauchy.  For every nonzero $v$, the sequence
$\log\alpha_n(v)$ is Cauchy in $\R$; define
$\alpha(v)=\lim_n\alpha_n(v)$ and set $\alpha(0)=0$.  Homogeneity and
the ultrametric inequality pass to the limit.

For $\varepsilon>0$, choose $N$ such that
$\dinf(\alpha_n,\alpha_m)\leq\varepsilon$ for $m,n\geq N$.  Letting
$m$ tend to infinity in \eqref{eq:comparison}, with $\alpha = \alpha_m$ and $\beta = \alpha_n$, gives
\[
  e^{-\varepsilon}\alpha_n\leq\alpha\leq
  e^{\varepsilon}\alpha_n \qquad(n\geq N).
\]
In particular $\alpha$ is nondegenerate and
$\dinf(\alpha_n,\alpha)\leq\varepsilon$.  Hence $\alpha_n\to\alpha$ in
$\Norm_{\alpha_0}(V)$.
\end{proof}

\begin{proposition}[Bounded suprema]\label{prop:supremum}
Let $\mathcal A\subseteq\Norm_{\alpha_0}(V)$ and suppose that one
constant $R$ satisfies
\[
  \dinf(\gamma,\alpha_0)\leq R
  \qquad(\gamma\in\mathcal A).
\]
Then
\[
  \beta(v)=\sup_{\gamma\in\mathcal A}\gamma(v)
\]
is a norm in $\Norm_{\alpha_0}(V)$ and
$\dinf(\beta,\alpha_0)\leq R$.
\end{proposition}

\begin{proof}
The uniform comparison with $\alpha_0$ makes the supremum finite and
nonzero away from the origin.  Homogeneity is immediate, and
\begin{align*}
 \beta(v+w)
 &\leq\sup_{\gamma\in\mathcal A}
       \max\{\gamma(v),\gamma(w)\}\\
 &\leq\max\{\beta(v),\beta(w)\}.
\end{align*}
The same comparison proves the asserted distance bound.
\end{proof}

\subsection{Hyperconvexity and exact circumcentres}
\label{subsec:hyperconvex-norm-space}

The order structure on norms gives more geometry than completeness.

\begin{theorem}[Hyperconvexity of a finite-distance component]
\label{thm:hyperconvex-component}
The metric space
\[
 \bigl(\Norm_{\alpha_0}(V),\dinf\bigr)
\]
is hyperconvex.  Explicitly, let
\[
 \{B(\alpha_i,r_i)\}_{i\in I}
\]
be any family of closed balls satisfying
\[
 \dinf(\alpha_i,\alpha_j)\leq r_i+r_j
 \qquad(i,j\in I).
\]
Then
\begin{equation}\label{eq:hyperconvex-center}
 \beta(v)
 =
 \sup_{i\in I}e^{-r_i}\alpha_i(v)
\end{equation}
belongs to every ball \(B(\alpha_i,r_i)\).

Consequently the component is an injective metric space, by the
Aronszajn--Panitchpakdi characterization of injective metric spaces
\cite{AronszajnPanitchpakdi1956}.
\end{theorem}

\begin{proof}
Fix \(i_0\in I\).  For every \(i\),
\[
 \alpha_i\leq e^{r_i+r_{i_0}}\alpha_{i_0},
\]
hence
\[
 e^{-r_i}\alpha_i\leq e^{r_{i_0}}\alpha_{i_0}.
\]
Thus the supremum in \eqref{eq:hyperconvex-center} is finite.  It is a
nondegenerate ultrametric norm because a pointwise supremum of ultrametric
seminorms is ultrametric, and it dominates the nondegenerate norm
\(e^{-r_{i_0}}\alpha_{i_0}\).

For fixed \(i\),
\[
 \beta\geq e^{-r_i}\alpha_i.
\]
Moreover, for every \(j\),
\[
 \alpha_j\leq e^{r_i+r_j}\alpha_i,
\]
so
\[
 e^{-r_j}\alpha_j\leq e^{r_i}\alpha_i.
\]
Taking the supremum over \(j\) gives
\[
 e^{-r_i}\alpha_i\leq\beta\leq e^{r_i}\alpha_i.
\]
Hence \(\dinf(\beta,\alpha_i)\leq r_i\).
\end{proof}

In finite dimension, Haettel proves the corresponding Helly property for
balls in the Goldman--Iwahori metric and relates it to injective
\(\ell^\infty\)-metrics on Bruhat--Tits buildings
\cite{Haettel2022}.  The preceding proof shows directly that the
finite-distance component used here retains the full hyperconvex
ball-intersection property even in the present infinite-dimensional setting.

\begin{theorem}[Exact circumradius of a bounded family]
\label{thm:exact-circumradius}
Let \(\mathcal A\subseteq\Norm_{\alpha_0}(V)\) be nonempty and bounded, and
put
\[
 D=\operatorname{diam}(\mathcal A).
\]
Then
\begin{equation}\label{eq:explicit-circumcenter}
 c_{\mathcal A}(v)
 =
 e^{-D/2}
 \sup_{\gamma\in\mathcal A}\gamma(v)
\end{equation}
satisfies
\[
 \sup_{\gamma\in\mathcal A}
 \dinf(c_{\mathcal A},\gamma)
 \leq\frac D2.
\]
Moreover \(D/2\) is the smallest possible radius of a ball containing
\(\mathcal A\).  Hence every bounded family has circumradius exactly half
its diameter.
\end{theorem}

\begin{proof}
Fix \(\gamma\in\mathcal A\).  Since
\[
 \eta\leq e^D\gamma
 \qquad(\eta\in\mathcal A),
\]
we have
\[
 e^{-D/2}\gamma
 \leq
 c_{\mathcal A}
 \leq
 e^{D/2}\gamma.
\]
Thus the displayed radius is at most \(D/2\).

Conversely, if \(\mathcal A\subseteq B(\beta,R)\), then for
\(\gamma,\eta\in\mathcal A\),
\[
 \dinf(\gamma,\eta)
 \leq
 \dinf(\gamma,\beta)+\dinf(\beta,\eta)
 \leq2R.
\]
Hence \(D\leq2R\).
\end{proof}

\subsection{Stable displacement}
\label{subsec:stable-displacement}

Let \(T\) be an isometry of a metric space.  Subadditivity gives
\begin{equation}\label{eq:stable-translation-length}
 \tau(T)
 =
 \lim_{m\to\infty}\frac1m d(x,T^mx)
 =
 \inf_{m\geq1}\frac1m d(x,T^mx),
\end{equation}
and the value is independent of \(x\).

\begin{proposition}[Change of base norm]
\label{prop:stable-change-base}
For \(\alpha,\beta\in\Norm_{\alpha_0}(V)\), \(g\in G\), and \(m\in\mathbb Z\),
\begin{equation}\label{eq:change-base-displacement}
 \dinf(\alpha,g^m\alpha)
 \leq
 2\dinf(\alpha,\beta)
 +
 \dinf(\beta,g^m\beta).
\end{equation}
Thus every estimate at an adapted norm \(\beta\) transfers to the reference
norm with an additive error independent of \(m\).
\end{proposition}

\begin{proof}
Apply the triangle inequality and the fact that the \(G\)-action is
isometric:
\[
 \dinf(\alpha,g^m\alpha)
 \leq
 \dinf(\alpha,\beta)
 +
 \dinf(\beta,g^m\beta)
 +
 \dinf(g^m\beta,g^m\alpha).
\]
\end{proof}

On a finite-dimensional \(E\)-vector space the stable length is spectral.

\begin{proposition}[Finite-dimensional spectral displacement]
\label{prop:finite-spectral-displacement}
Let \(W\) be finite-dimensional, \(T\in\GL(W)\), and let \(\alpha\) be any
norm on \(W\).  If \(r(T)\) denotes the spectral radius, then
\begin{equation}\label{eq:finite-spectral-length}
 \tau(T)
 =
 \max\{\log r(T),\log r(T^{-1})\}.
\end{equation}
After extending scalars to an algebraic closure, if
\(\lambda_1,\ldots,\lambda_d\) are the eigenvalues of \(T\), then
\begin{equation}\label{eq:finite-spectral-length-eigenvalues}
 \tau(T)
 =
 \max_i\left|\log|\lambda_i|\right|.
\end{equation}
If every eigenvalue has absolute value \(1\), there is a full
\(\Ocoeff\)-lattice stable under both \(T\) and \(T^{-1}\); its gauge norm
makes \(T\) an isometry.
\end{proposition}

\begin{proof}
The identity
\[
 \dinf(\alpha,T^m\alpha)
 =
 \max\{\log\|T^m\|_\alpha,\log\|T^{-m}\|_\alpha\}
\]
and the finite-dimensional spectral-radius formula give
\eqref{eq:finite-spectral-length}.  The eigenvalue formula follows after
scalar extension.

If all eigenvalues have absolute value \(1\), then both \(T\) and \(T^{-1}\)
are integral over \(\Ocoeff\).  Hence \(\Ocoeff[T,T^{-1}]\) is finite over
\(\Ocoeff\).  For any full lattice \(\Lambda_0\),
\[
 \Lambda=\Ocoeff[T,T^{-1}]\Lambda_0
\]
is a full lattice stable under \(T\) and \(T^{-1}\).
\end{proof}

\subsection{Why the sup metric is not CAT(0)}

Suppose that $V$ has two linearly independent basis vectors $e_1,e_2$.
Keeping all other basis weights fixed, define diagonal norms by
\[
  \alpha_x(a_1e_1+a_2e_2+\cdots)
   =\max\{|a_1|e^{x_1},|a_2|e^{x_2},\ldots\},
  \qquad x=(x_1,x_2)\in\R^2.
\]
Then
\[
  \dinf(\alpha_x,\alpha_y)
    =\max\{|x_1-y_1|,|x_2-y_2|\}.
\]
Consequently this diagonal apartment carries the $\ell^\infty$ metric.
Between $(0,0)$ and $(1,0)$, each $(1/2,t)$ with
$|t|\leq1/2$ is a midpoint and lies on a geodesic.  Hence the space is
not uniquely geodesic and is not CAT(0).

This does not indicate negative behaviour.  In finite dimension the
Goldman--\allowbreak Iwahori metric has strong injectivity and Helly
properties \cite{Haettel2022}.  Injective metric spaces share several
fixed-point features with CAT(0) spaces \cite{Lang2013}.  For the
particular fixed-point statement required here, Proposition
\ref{prop:supremum} gives a direct proof.

\section{Bounded orbits and invariant norms}
\label{sec:bounded}

The group $G$ acts on norms by
\begin{equation}\label{eq:group-action}
  (g\alpha)(v)=\alpha(g^{-1}v).
\end{equation}
This action preserves $\dinf$ whenever the distance is finite.

\begin{theorem}[Bounded-orbit criterion]\label{thm:bounded-orbit}
Let $\alpha_0$ be a norm on $V$.  The following are equivalent.
\begin{enumerate}
  \item There is a $G$-invariant norm in $\Norm_{\alpha_0}(V)$.
  \item The orbit $G\alpha_0$ is bounded in $\Norm_{\alpha_0}(V)$.
  \item The orbit of some norm in $\Norm_{\alpha_0}(V)$ is bounded.
\end{enumerate}
If these conditions hold, then
\begin{equation}\label{eq:invariant-sup}
  \alpha_G(v)=\sup_{g\in G}\alpha_0(g^{-1}v)
\end{equation}
is a $G$-invariant norm in $\Norm_{\alpha_0}(V)$.
\end{theorem}

\begin{proof}
Suppose first that $\beta\in\Norm_{\alpha_0}(V)$ is $G$-invariant.  By
the triangle inequality and invariance,
\[
 \dinf(\alpha_0,g\alpha_0)
 \leq\dinf(\alpha_0,\beta)+\dinf(\beta,g\alpha_0)
 =2\dinf(\alpha_0,\beta),
\]
so the orbit is bounded.

Conversely, suppose that
\[
 D=\sup_{g\in G}\dinf(\alpha_0,g\alpha_0)<\infty.
\]
Then every $g\alpha_0$ lies between $e^{-D}\alpha_0$ and
$e^D\alpha_0$.  Proposition \ref{prop:supremum} therefore shows that
\eqref{eq:invariant-sup} is a norm in the same component.  For
$h\in G$,
\[
 \alpha_G(h^{-1}v)
 =\sup_{g\in G}\alpha_0(g^{-1}h^{-1}v)
 =\sup_{u\in G}\alpha_0(u^{-1}v)
 =\alpha_G(v),
\]
so it is invariant.  The equivalence with (3) follows because all
points in one finite-distance component are at finite distance from one
another.
\end{proof}

\begin{remark}[The least invariant norm above a basepoint]
If the orbit of $\alpha_0$ is bounded, then $\alpha_G$ is the least
$G$-invariant norm dominating $\alpha_0$.  Indeed, if $\beta$ is
invariant and $\beta\geq\alpha_0$, then
$\beta(v)=\beta(g^{-1}v)\geq\alpha_0(g^{-1}v)$ for every $g$, hence
$\beta\geq\alpha_G$.
\end{remark}

\begin{theorem}[Exact distance to the invariant locus]
\label{thm:exact-orbit-radius}
Assume \(G\alpha_0\) is bounded and put
\[
 D
 =
 \operatorname{diam}(G\alpha_0)
 =
 \sup_{g\in G}\dinf(\alpha_0,g\alpha_0).
\]
Let \(\alpha_G\) be the invariant supremum norm constructed above.  Then
\begin{equation}\label{eq:optimal-invariant-center}
 \alpha_*=e^{-D/2}\alpha_G
\end{equation}
is \(G\)-invariant and
\begin{equation}\label{eq:exact-distance-fixed-locus}
 \boxed{
 \inf_{\beta\in\Norm_{\alpha_0}(V)^G}
 \dinf(\alpha_0,\beta)
 =
 \dinf(\alpha_0,\alpha_*)
 =
 \frac D2.
 }
\end{equation}
\end{theorem}

\begin{proof}
Isometricity gives
\[
 \operatorname{diam}(G\alpha_0)
 =
 \sup_{g\in G}\dinf(\alpha_0,g\alpha_0).
\]
Theorem~\ref{thm:exact-circumradius}, applied to the orbit, shows that
\(\alpha_*\) is at distance at most \(D/2\) from every orbit point.  Since
the supremum norm \(\alpha_G\) is \(G\)-invariant, so is \(\alpha_*\).

Conversely, if \(\beta\) is \(G\)-invariant, then
\[
 \dinf(\alpha_0,g\alpha_0)
 \leq
 \dinf(\alpha_0,\beta)+
 \dinf(\beta,g\alpha_0)
 =
 2\dinf(\alpha_0,\beta).
\]
Taking the supremum gives the reverse inequality.
\end{proof}

\begin{corollary}[Sharp quantitative bounded-orbit criterion]
\label{cor:sharp-bounded-orbit}
For \(R\geq0\), the following are equivalent:
\begin{enumerate}
 \item
 \[
  \sup_{g\in G}\dinf(\alpha_0,g\alpha_0)\leq2R;
 \]
 \item there is a \(G\)-invariant norm \(\beta\) with
 \[
  \dinf(\alpha_0,\beta)\leq R.
 \]
\end{enumerate}
Thus the invariant-norm radius is exactly half the orbit diameter.
\end{corollary}

\begin{corollary}[Compatible finite-level centre]
\label{cor:compatible-finite-level-center}
With the notation above, write
\[
 \alpha_*=(\alpha_{*,K})_K.
\]
Then
\[
 d_K(\alpha_{0,K},\alpha_{*,K})\leq D/2
\]
for every \(K\), and the family \((\alpha_{*,K})_K\) is compatible under
restriction and transport.  Thus the optimal global centre is already a
compatible centre of every finite transport diagram.
\end{corollary}

\begin{remark}[There is no fixed point theorem]
Observe that if $G$ stabilizes a non-empty bounded subset $B$ of $\Norm_{\alpha_0}(V)$, then the fixed point of $G$(if it exists) may not lie in $B$. For example, if $G=GL_3(F)$, define a smooth character $\varepsilon(g)=(-1)^{v_F(\det g)}$.
Let $V=Ee_1\oplus Ee_2$ with \(G\)-action
\[
\rho(g)=
\begin{cases}
	1, & \text{if } \varepsilon(g)=1,\\
	s, & \text{if } \varepsilon(g)=-1,
\end{cases}
\qquad
s(e_1)=e_2,\quad s(e_2)=e_1.
\]
Equivalently, $V\simeq \mathbf{1}\oplus\varepsilon$.This representation is locally algebraic with trivial algebraic factor. Choose \(c>1\) and define
\[
\alpha_1(xe_1+ye_2)
=\max\bigl\{|x|,c|y|\bigr\},
\qquad
\alpha_2(xe_1+ye_2)
=\max\bigl\{c|x|,|y|\bigr\}.
\]
If \(v_F(\det g)\) is even, then $g\alpha_1=\alpha_1$ and $g\alpha_2=\alpha_2$. If \(v_F(\det g)\) is odd, then $g\alpha_1=\alpha_2$ and
$g\alpha_2=\alpha_1$. In particular, $g_0=\operatorname{diag}(\varpi_F,1,1)$ exchanges the two norms. Hence $B=\{\alpha_1,\alpha_2\}$ is \(G\)-stable and bounded, since $d_\infty(\alpha_1,\alpha_2)=\log c$. Nevertheless, \(B\) has no \(G\)-fixed point, because \(g_0\) exchanges
its two elements. The invariant supremum norm is
\[
\begin{aligned}
	\alpha_B(xe_1+ye_2)
	&=\sup_{\gamma\in B}\gamma(xe_1+ye_2)\\
	&=c\max\bigl\{|x|,|y|\bigr\},
\end{aligned}
\]
but $\alpha_B\notin B$. So the Bruhat-Tits fixed point theorem is does not hold in $\Norm_{\alpha_0}(V)$, for example \cite[Theorem 11.23]{MR2439729}.

\end{remark}

\begin{corollary}\label{cor:center}
Suppose that $V$ has central character $\omega:Z(G)\to E^\times$.  If
$V$ admits a $G$-invariant norm, then $\omega$ is unitary:
$|\omega(z)|=1$ for every $z\in Z(G)$.
\end{corollary}

\begin{proof}
For $v\neq0$ and $z\in Z(G)$,
\[
 \alpha(v)=\alpha(zv)=\alpha(\omega(z)v)
 =|\omega(z)|\alpha(v).
\]
\end{proof}

\section{The Emerton's condition and the existence of G-invariant norms}

Let $P=MN$ be a parabolic subgroup, fix a compact open subgroup
$N_0\subseteq N$, and put
\[
 M^+=\{m\in M:mN_0m^{-1}\subseteq N_0\},\qquad
 Z_M^+=Z(M)\cap M^+.
\]
Write
\[
 \delta_P(m)
  =\left|\det\!\left(\operatorname{Ad}(m)|_{\operatorname{Lie}N}
    \right)\right|_F.
\]
We use Emerton's finite-slope Jacquet module $J_P(V)$ for the locally
algebraic representation $V$ \cite{MR2181093}.  A character
$\chi:Z(M)\to E^\times$ is called an \emph{exponent} if its
generalised eigenspace in $J_P(V)$ is nonzero.

\begin{definition}[Emerton's condition]
\label{def:Emerton-condition}
With the preceding unnormalised Jacquet convention, $V$ satisfies
$(\mathrm{EC})$ if, for every $P,N_0$, every exponent $\chi$, and every
$z\in Z_M^+$,
\begin{equation}\label{eq:Emerton-condition}
 \bigl|\chi(z)\delta_P(z)^{-1}\bigr|\leq1.
\end{equation}
\end{definition}

The factor $\delta_P^{-1}$ is part of the convention: on the
$\chi$-eigenspace the positive Jacquet operator is normalised to have
eigenvalue $\chi(z)\delta_P(z)^{-1}$.  With a normalised Jacquet
functor the same condition is written after absorbing the appropriate
modulus character into $\chi$.  Hu proves that an invariant norm
implies these inequalities and relates them to admissible filtrations
\cite{MR2560407}; the inequalities alone do not construct a global
lattice.

\subsection{A general strategy}

We now explain how the preceding necessary conditions fit into a general
strategy for constructing an invariant norm.  In this subsection $G$ is the
$F$-points of a connected reductive group, $V$ is locally algebraic, the
central character is assumed unitary, and $V$ is assumed to satisfy
$(\mathrm{EC})$.  The discussion is a reduction programme rather than a
sufficiency theorem: the missing point in general is to construct a basis in
which the required matrix estimates descend to a finite transition problem.

Choose a special maximal compact subgroup $K\subseteq G$ and a maximal
$F$-split torus $S$.  After fixing a positive chamber, Cartan decomposition
has the form
\[
 G=\coprod_{\lambda\in\Lambda^+}K\lambda(\varpi_F)K,
\]
where $\Lambda^+$ is the semigroup of dominant cocharacters modulo the
finite ambiguity coming from the component group and the centre.  Choose
cocharacters $\lambda_1,\ldots,\lambda_s$ which generate the image of
$\Lambda^+$ modulo the central cocharacters, and put
\[
 a_i=\lambda_i(\varpi_F).
\]
Because the central character is unitary, central factors act isometrically
on every norm.  Hence, for a $K$-invariant norm $\alpha$, a uniform estimate
\begin{equation}\label{eq:general-Cartan-uniform-target}
 \sup_{m\geq0}d_\infty(\alpha,a_i^m\alpha)<\infty,
 \qquad 1\leq i\leq s,
\end{equation}
(and the analogous estimates for the inverse directions, or a reduction of
them to complementary positive directions) implies boundedness of the full
Cartan orbit.  Indeed, the action on the norm space is isometric, so the
triangle inequality bounds a product
$a_1^{m_1}\cdots a_s^{m_s}$ by the sum of the finitely many bounds in
\eqref{eq:general-Cartan-uniform-target}.  Compact and unitary central
factors do not change the bound.

The useful algebraic input is a Hamel basis adapted simultaneously to the
smooth type data and to a weight basis of the algebraic factor.  Write such a
basis as $\{e_\nu\}_{\nu\in I}$ and
\begin{equation}\label{eq:general-matrix-coefficients}
 ge_\nu=\sum_{\mu\in I}A_{\mu\nu}(g)e_\mu,
\end{equation}
where every column is finite.  Choose positive weights $c_\nu$ and the max
norm
\[
 \alpha_x\!\left(\sum_\nu b_\nu e_\nu\right)
 =\max_\nu |b_\nu|c_\nu,
 \qquad
 x_\nu=\log_{q_E}c_\nu.
\]
Then the matrix of $g$ gives
\begin{equation}\label{eq:general-weighted-matrix-bound}
 \|g\|_{\alpha_x}
 =\sup_{\mu,\nu}
 |A_{\mu\nu}(g)|\frac{c_\mu}{c_\nu},
\end{equation}
with the same formula for $g^{-1}$.  Thus the problem of bounding
$d_\infty(\alpha_x,g\alpha_x)$ becomes a problem of finding a constant $C$ independent of $\mu$,$\nu$, and $g$, such that $v_E(A_{\mu\nu}(g)) + x_{\mu} - x_{\nu} \geq C$.

\subsection{Transition graphs and finite-state reductions}

The matrix estimate in \eqref{eq:general-weighted-matrix-bound} admits a
precise graph-theoretic formulation.  We first isolate the part which is
purely linear algebra and therefore unconditional.

\begin{definition}[Weighted Cartan transition graph]
\label{def:weighted-Cartan-transition-graph}
Let $b\in G$ and let $\{e_\nu\}_{\nu\in I}$ be a Hamel basis for which
\[
 be_\nu=\sum_{\mu\in I}A_{\mu\nu}(b)e_\mu
\]
has finite columns.  The \emph{weighted transition graph} $\Gamma_b(I)$ has
vertex set $I$ and a directed edge
\[
 e=(\nu\xrightarrow{\,b\,}\mu)
\]
precisely when $A_{\mu\nu}(b)\neq0$.  Its edge cost is
\[
 w_b(e)=v_E\!\left(A_{\mu\nu}(b)\right).
\]
If a normalized operator $\widetilde b=\gamma_b b$ is used instead of $b$,
then $A_{\mu\nu}(b)$ and $w_b(e)$ mean the matrix coefficient and valuation
for $\widetilde b$.  Thus all algebraic-weight and modulus factors are part
of the edge cost rather than being added informally afterwards.
\end{definition}

\begin{lemma}[Difference inequalities and weighted max norms]
\label{lem:graph-potential-max-norm}
Let $x:I\to\mathbb R$, put $c_\nu=q_E^{x_\nu}$, and let $\alpha_x$ be the
corresponding weighted max norm.  For a fixed $b\in G$ the following are
equivalent:
\begin{enumerate}
 \item $\|b\|_{\alpha_x}\leq1$;
 \item for every edge $\nu\xrightarrow{b}\mu$ of $\Gamma_b(I)$,
 \begin{equation}\label{eq:edge-difference-inequality-general}
   x_\mu-x_\nu\leq w_b(\nu\to\mu).
 \end{equation}
\end{enumerate}
Consequently, if \eqref{eq:edge-difference-inequality-general} holds for
$b$ and for $b^{-1}$, then $b$ acts isometrically on $\alpha_x$.  If it
holds for every element of a semigroup generated by
$b_1,\ldots,b_s$ at the level of the generators, then every positive word in
the $b_i$ is nonexpanding.
\end{lemma}

\begin{proof}
By \eqref{eq:general-weighted-matrix-bound},
\[
 \|b\|_{\alpha_x}
 =\sup_{\mu,\nu}
 q_E^{-v_E(A_{\mu\nu}(b))+x_\mu-x_\nu}.
\]
The supremum is at most $1$ exactly when
$x_\mu-x_\nu\leq v_E(A_{\mu\nu}(b))$ for every nonzero coefficient.  This is
\eqref{eq:edge-difference-inequality-general}.  Applying the same argument
to $b^{-1}$ gives both inequalities
$\alpha_x(bv)\leq\alpha_x(v)$ and
$\alpha_x(v)\leq\alpha_x(bv)$, hence equality.  The last assertion follows
from submultiplicativity of operator norms.
\end{proof}

The graph in Definition~\ref{def:weighted-Cartan-transition-graph} is
usually infinite.  The phrase ``finite type--Bruhat graph'' will mean the
following concrete finite quotient, not merely a heuristic periodicity.

\begin{definition}[Finite type--Bruhat quotient]
\label{def:finite-type-Bruhat-quotient}
Fix a finite set $\mathcal A\subseteq G$ of Cartan generators and, when
needed, their inverses.  A \emph{finite type--Bruhat quotient} of the family
$\{\Gamma_b(I)\}_{b\in\mathcal A}$ consists of
\begin{enumerate}
 \item a finite directed graph $\overline\Gamma$ whose edges are labelled by
 elements $b\in\mathcal A$;
 \item a surjective state map $\pi:I\to V(\overline\Gamma)$;
 \item for every labelled edge
 $\bar e=(u\xrightarrow{b}v)$ a real number $\bar w_b(\bar e)$;
\end{enumerate}
with the property that every edge
$e=(\nu\xrightarrow{b}\mu)$ in the original graph maps to an edge
$\bar e=(\pi(\nu)\xrightarrow{b}\pi(\mu))$ and satisfies the uniform lower
bound
\begin{equation}\label{eq:finite-quotient-cost-descent}
 \bar w_b(\bar e)\leq w_b(e).
\end{equation}
The quotient is called \emph{exact} when the quotient cost is the minimum of
the costs of all lifted edges with the same finite states.
\end{definition}

Thus, in order to prove that the transition data descend to a finite
type--Bruhat graph, one must prove three separate facts: first, that basis
indices admit finitely many relevant type--Bruhat states after the chosen
periodic or central identifications; second, that the support of every
Cartan matrix coefficient is determined by those states; and third, that the
valuations of all lifted coefficients admit the state-dependent uniform
lower bounds \eqref{eq:finite-quotient-cost-descent}.  For the metric
$d_\infty$ the same assertions must also be available for the inverse Cartan
directions, unless those directions are reduced to positive ones by compact
and unitary central factors.  None of these three facts follows from
Emerton's condition alone.

\begin{proposition}[Finite negative-cycle criterion]
\label{prop:finite-negative-cycle-criterion}
Let $\overline\Gamma$ be a finite directed graph with real edge costs
$\bar w(e)$.  There exists a function
$\bar x:V(\overline\Gamma)\to\mathbb R$ satisfying
\begin{equation}\label{eq:finite-difference-system}
 \bar x(v)-\bar x(u)\leq\bar w(u\to v)
\end{equation}
for every directed edge if and only if every directed cycle has nonnegative
total cost.  If $\overline\Gamma$ is a finite type--Bruhat quotient, then
$x_\nu=\bar x(\pi(\nu))$ satisfies all the lifted inequalities
\eqref{eq:edge-difference-inequality-general}.
\end{proposition}

\begin{proof}
If \eqref{eq:finite-difference-system} holds, summing it around a directed
cycle telescopes the left-hand side to $0$, so the total cycle cost is
nonnegative.  Conversely, adjoin a new source $s$ with a zero-cost edge to
every vertex.  Because there is no negative directed cycle, the minimum cost
$d(v)$ of a directed path from $s$ to $v$ is finite and is attained by a
simple path.  For every edge $u\to v$ one has
$d(v)\leq d(u)+\bar w(u\to v)$, so $\bar x=d$ satisfies
\eqref{eq:finite-difference-system}.  Finally,
\[
 x_\mu-x_\nu
 =\bar x(\pi(\mu))-\bar x(\pi(\nu))
 \leq\bar w(\pi(\nu)\to\pi(\mu))
 \leq w_b(\nu\to\mu),
\]
which is the lifted inequality.
\end{proof}

\begin{corollary}[Finite-state sufficient criterion]
\label{cor:finite-state-Cartan-sufficient}
Suppose the transition graphs for the chosen Cartan generators and the
required inverse directions admit a common finite type--Bruhat quotient and
that every directed cycle in that quotient has nonnegative cost.  Then the
lifted potential of Proposition~\ref{prop:finite-negative-cycle-criterion}
defines a weighted Hamel norm for which all those Cartan generators act
isometrically.  If the norm is invariant under the chosen maximal compact
subgroup and the central character is unitary, the full $G$-orbit of the
norm is bounded; hence Theorem~\ref{thm:bounded-orbit} produces a
$G$-invariant norm.
\end{corollary}

\begin{proof}
The proposition and Lemma~\ref{lem:graph-potential-max-norm} give
isometry for the selected Cartan directions.  Cartan decomposition expresses
every $g\in G$ as compact factors, a product of those Cartan directions, and
a central factor.  Compact invariance and Lemma~\ref{lem:unitary-central-direction}
below remove the first and last factors.  Thus the orbit is bounded, and
Theorem~\ref{thm:bounded-orbit} applies.
\end{proof}

\subsection{What Emerton's condition controls}

We next make precise the relation with Jacquet modules.  This relation is
spectral; by itself it does not identify the individual matrix coefficients
of a Cartan translate.

\begin{definition}[Positive Jacquet translation direction]
\label{def:Jacquet-translation-direction}
Let $P=MN$ and $N_0\subseteq N$ be as above.  A \emph{positive Jacquet
translation} is an element $z\in Z_M^+$.  Its \emph{noncentral translation
direction} is its class modulo $Z(G)\cap Z_M^+$.  On the finite-slope
Jacquet module we use the normalized positive operator for which the
$\chi$-generalized eigenspace has eigenvalue
\[
 \lambda_{\chi,z}=\chi(z)\delta_P(z)^{-1}.
\]
Thus a translation direction refers to a ray in the positive semigroup
$Z_M^+/(Z(G)\cap Z_M^+)$, not to an individual edge of a Hamel-basis graph.
\end{definition}

\begin{lemma}[Spectral consequence of Emerton's condition]
\label{lem:EC-spectral-translation}
Assume $(\mathrm{EC})$.  For every exponent $\chi$ and every
$z\in Z_M^+$,
\[
 v_E\!\left(\lambda_{\chi,z}\right)
 =v_E\!\left(\chi(z)\delta_P(z)^{-1}\right)\geq0.
\]
Equivalently, every eigenvalue of the normalized positive Jacquet operator
on a finite-dimensional generalized exponent space has absolute value at
most one.
\end{lemma}

\begin{proof}
This is exactly Definition~\ref{def:Emerton-condition}, rewritten using
$|a|=q_E^{-v_E(a)}$.
\end{proof}

This is the precise sense in which positive Cartan translations are
``detected on Jacquet modules'': after passing to $N_0$-coinvariants and then
to the finite-slope part, the positive double-coset operator attached to
$z$ acts on each exponent space with the normalized eigenvalue
$\chi(z)\delta_P(z)^{-1}$.  What is \emph{not} automatic is a lower bound for
every coefficient of the individual group element $z$, or for every summand
in a double-coset average.  Cancellation in the Jacquet operator can hide
large individual summands.  Bridging this average-to-translate gap is one of
the genuine obstacles in passing from $(\mathrm{EC})$ to a bounded
$G$-orbit.

\begin{lemma}[Unitary central directions]
\label{lem:unitary-central-direction}
Suppose the central character $\omega_V:Z(G)\to E^\times$ is unitary.  Then
every $z\in Z(G)$ fixes every norm on $V$ under the induced action on norm
space.  In particular
\[
 d_\infty(\alpha,z\alpha)=0
\]
for every norm $\alpha$.
\end{lemma}

\begin{proof}
Since $z$ acts on $V$ by the scalar $\omega_V(z)$ and
$|\omega_V(z)|=1$, homogeneity of a norm gives
$\alpha(zv)=|\omega_V(z)|\alpha(v)=\alpha(v)$ for every $v$.
\end{proof}

Thus the statement that ``the unitary central character removes the central
direction'' means exactly that central cocharacters contribute zero metric
displacement.  They may therefore be factored out of the Cartan semigroup
without creating growth.  This does not remove noncentral negative
translation directions.

\begin{definition}[Jacquet-compatible finite transition model]
\label{def:Jacquet-compatible-transition-model}
A finite type--Bruhat quotient is called \emph{Jacquet-compatible} if, for
every directed cycle $C$, its total cost admits an identity
\begin{equation}\label{eq:cycle-cost-decomposition}
 \operatorname{cost}(C)
 =\sum_{r=1}^a
   v_E\!\left(\chi_r(z_r)\delta_{P_r}(z_r)^{-1}\right)
  +\sum_{t=1}^b \beta_t,
 \qquad \beta_t\geq0,
\end{equation}
where $z_r\in Z_{M_r}^+$ are positive Jacquet translations.  Central
factors are omitted from \eqref{eq:cycle-cost-decomposition}, since
Lemma~\ref{lem:unitary-central-direction} shows that their cost is zero.
The numbers $\beta_t$ record finite-Weyl, panel, or type-intertwining
reductions whose matrices preserve specified integral lattices.
\end{definition}

\begin{proposition}[Conditional passage from arithmetic to the finite graph]
\label{prop:Jacquet-compatible-no-negative-cycle}
Assume $(\mathrm{EC})$, a unitary central character, and the existence of a
Jacquet-compatible finite transition model.  Then the finite graph has no
negative directed cycle.  Hence it admits a potential and yields the Cartan
bounds of Corollary~\ref{cor:finite-state-Cartan-sufficient}.
\end{proposition}

\begin{proof}
Every term in the first sum of
\eqref{eq:cycle-cost-decomposition} is nonnegative by
Lemma~\ref{lem:EC-spectral-translation}; every $\beta_t$ is nonnegative by
definition.  Hence every cycle has nonnegative total cost.  Apply
Proposition~\ref{prop:finite-negative-cycle-criterion} and then
Corollary~\ref{cor:finite-state-Cartan-sufficient}.
\end{proof}

The proposition deliberately isolates the missing representation-theoretic
input.  For a general connected reductive group, $(\mathrm{EC})$ does not
construct a finite type--Bruhat quotient, does not prove the uniform lower
bound \eqref{eq:finite-quotient-cost-descent}, and does not prove the exact
cycle decomposition \eqref{eq:cycle-cost-decomposition}.  Even when stable
lattices are known in Jacquet or type modules, one must still lift them to
$V$ and prove compatibility with a single commensurability constant at all
compact-open levels.  These are the obstacles to turning the preceding
conditional proposition into a general sufficiency theorem.

\section{\texorpdfstring{The case $G=\GL_n(F)$}
{The case G=GLn(F)}}
\label{sec:GLn}

Throughout this section
\[
 G=\GL_n(F),
 \qquad
 V=\operatorname{BS}(r)
 =\pi_{\mathrm{gen}}(r)\otimes_E\pi_{\mathrm{alg}}(r),
\]
where $r:G_F\to\GL_n(E)$ is potentially semistable.  Fontaine's module associated to $r$ is equipped with Frobenius, monodromy $N$, and descent data. From that data we extract the Weil--Deligne representation
\[
 D=\mathrm{WD}(r)=(\tau,N),
\]
with the chosen Frobenius-semisimplification convention.

We use the $p$-adic local Langlands normalization appearing in
\cite{Pyvovarov2021}: if $\operatorname{rec}_p$ is the $p$-adic form of the
classical local Langlands correspondence, put
\[
 r_p(\pi)=\operatorname{rec}_p\bigl(\pi\otimes|\det|_F\bigr),
 \qquad
 \pi_{\mathrm{sm}}(r)
 =r_p^{-1}\bigl(\mathrm{WD}(r)^{F\text{-}\mathrm{ss}}\bigr).
\]
The representation $\pi_{\mathrm{gen}}(r)$ is the unique standard module,
formed using \emph{normalized} parabolic induction, which surjects onto
$\pi_{\mathrm{sm}}(r)$.  Accordingly, throughout the rest of the paper
\[
 \pi_1\times\cdots\times\pi_t
 :=i_P^{\GL_n(F)}(\pi_1\boxtimes\cdots\boxtimes\pi_t)
\]
means normalized parabolic induction.

Decompose $D$ into indecomposable Weil--Deligne summands
\begin{equation}\label{eq:WD-indecomposable-decomposition}
 D=\bigoplus_{i=1}^t D_i,
 \qquad
 D_i\simeq D(\tau_i,\ell_i)
 =\tau_i\oplus\tau_i(1)\oplus\cdots\oplus\tau_i(\ell_i-1),
\end{equation}
where $\tau_i$ is irreducible and monodromy maps successive twists along one
Jordan chain.  Let $\rho_i$ be the supercuspidal representation corresponding
to $\tau_i$ under this normalization and write $\nu=|\det|_F$.  The summand
$D_i$ determines the segment
\begin{equation}\label{eq:WD-to-segment}
 \Delta_i=[\rho_i,\rho_i\nu,\ldots,\rho_i\nu^{\ell_i-1}],
\end{equation}
up to an overall unramified twist absorbed into $\rho_i$.  After ordering the
segments in standard order,
\begin{equation}\label{eq:pi-gen-from-WD}
 \pi_{\mathrm{gen}}(r)
 =Q(\Delta_1)\times\cdots\times Q(\Delta_t),
\end{equation}
where the symbol $\times$ has the normalized meaning fixed above.

For the algebraic factor, assume that $E$ contains the images of all
embeddings $\kappa:F\hookrightarrow E$.  Write the $\kappa$-Hodge--Tate weights as
\[
 \operatorname{HT}_\kappa(r)
 =\{i_{\kappa,1}<\cdots<i_{\kappa,n}\}.
\]
Relative to the upper-triangular Borel, define the dominant highest weight
\begin{equation}\label{eq:BS-highest-weight}
 \xi_\kappa
 =\bigl(\xi_{\kappa,1},\ldots,\xi_{\kappa,n}\bigr),
 \qquad
 \xi_{\kappa,j}=-i_{\kappa,j}+j-1.
\end{equation}
Let $\pi_{alg, \kappa}(r)$ be the irreducible algebraic representation with highest
weight $\xi_\kappa$,and
\begin{equation}\label{eq:pi-alg-from-HT}
 \pi_{\mathrm{alg}}(r)
 =\bigotimes_{\kappa:F\hookrightarrow E}\pi_{alg, \kappa}(r),
\end{equation}
with diagonal $\GL_n(F)$-action. Observe that the highest weight of $\pi_{\mathrm{alg}}(r)$ with respect to the upper triangular matrices is given by $ \mathrm{diag}(x_1,..., x_n) \mapsto \prod_{j=1}^{n}\prod_{\kappa:F\hookrightarrow E} \kappa(x_j^{\xi_{\kappa,j}})$

\subsection{Multisegments and semisimple types}

Let
\[
 \mathfrak m=\{\Delta_1,\ldots,\Delta_t\}
\]
be the Bernstein--Zelevinsky multisegment attached to the Weil--Deligne
representation of \(r\).  With our convention,
\begin{equation}\label{eq:multisegment-parabolic-induction}
 \pi_{\mathrm{gen}}(r)
 =
 Q(\Delta_1)\times\cdots\times Q(\Delta_t),
\end{equation}
the normalized parabolic induction of the segment representations.

Group the segments according to their cuspidal line:
\[
 \mathfrak m=\coprod_{\rho}\mathfrak m_\rho.
\]
The semisimple-type theory of Bushnell--Kutzko identifies the Bernstein block
with the module category of a tensor product of affine Hecke algebras, one
factor for each cuspidal line; parabolic induction and Jacquet functors are
compatible with the corresponding standard maps of Hecke algebras
\cite{BushnellKutzko1999}.

Fix a semisimple type
\[
 (J,\lambda)
\]
for the block of \(\pi_{\mathrm{gen}}(r)\), put
\[
 \mathcal P_\lambda=\cInd_J^G\lambda,
 \qquad
 \mathcal H=\End_G(\mathcal P_\lambda)^{\mathrm{op}},
\]
and let
\[
 M(\mathfrak m)=\Hom_J(\lambda,\pi_{\mathrm{gen}}(r)).
\]
Then evaluation gives
\begin{equation}\label{eq:type-evaluation-multisegment}
 \mathcal P_\lambda\otimes_{\mathcal H}M(\mathfrak m)
 \longrightarrow
 \pi_{\mathrm{gen}}(r).
\end{equation}
In the Bernstein block under consideration this is an isomorphism.

For clarity, first suppose all segments lie on one cuspidal line of a
supercuspidal representation \(\rho\) of \(\GL_d(F)\):
\[
 \Delta_i=
 [\rho\nu^{a_i},\rho\nu^{a_i+1},\ldots,\rho\nu^{b_i}],
 \qquad
 \ell_i=b_i-a_i+1.
\]
Put
\[
 N=\ell_1+\cdots+\ell_t,
 \qquad
 W=S_N,
 \qquad
 W_P=S_{\ell_1}\times\cdots\times S_{\ell_t}.
\]
The relevant affine Hecke algebra is of type \(A_{N-1}\), with parameter
\(q_\rho\), and the parabolic subalgebra is
\[
 \mathcal H_P
 \simeq
 \mathcal H_{\ell_1}\otimes\cdots\otimes\mathcal H_{\ell_t}.
\]
Each segment factor determines its sign/discrete-series character
\(\chi_{\Delta_i}\).

More precisely, choose Bernstein translation generators
$X_{i,1}^{\pm1},\ldots,X_{i,\ell_i}^{\pm1}$ in the affine Hecke factor
$\mathcal H_{\ell_i}$ attached to the cuspidal line of $\rho_i$.  The
segment $\Delta_i$ determines the one-dimensional sign/discrete-series
character
\begin{equation}\label{eq:segment-Hecke-character-definition}
 \chi_{\Delta_i}:\mathcal H_{\ell_i}\longrightarrow E
\end{equation}
characterized, in the normalization $(T_s+1)(T_s-q_\rho)=0$, by
\[
 \chi_{\Delta_i}(T_s)=-1
\]
for every finite simple reflection inside the $i$th segment and by
\begin{equation}\label{eq:segment-Bernstein-values-general}
 \chi_{\Delta_i}(X_{i,j})
 =z_{\Delta_i}q_\rho^{-(j-1)},
 \qquad 1\leq j\leq\ell_i,
\end{equation}
where $z_{\Delta_i}=\chi_{\Delta_i}(X_{i,1})$ records the initial
unramified parameter of the segment.  Equivalently, this is the affine-Hecke
character corresponding to the essentially square-integrable representation
$Q(\Delta_i)$ under the type equivalence.

The external tensor product
\[
 \chi_{\Delta_1}\boxtimes\cdots\boxtimes\chi_{\Delta_t}
 :\mathcal H_{\ell_1}\otimes\cdots\otimes\mathcal H_{\ell_t}
 \longrightarrow E
\]
is defined on pure tensors by
\begin{equation}\label{eq:external-segment-character}
 (\chi_{\Delta_1}\boxtimes\cdots\boxtimes\chi_{\Delta_t})
 (h_1\otimes\cdots\otimes h_t)
 =\prod_{i=1}^t\chi_{\Delta_i}(h_i),
\end{equation}
and extended $E$-linearly.  We abbreviate this character to
$\chi_{\mathfrak m}$.

Thus the standard multisegment Hecke module is
\begin{equation}\label{eq:multisegment-Hecke-standard-module}
 M(\mathfrak m)
 =
 \mathcal H_N
 \otimes_{\mathcal H_P}
 \left(
 \chi_{\Delta_1}\boxtimes\cdots\boxtimes\chi_{\Delta_t}
 \right).
\end{equation}
This is the affine-Hecke counterpart of
\eqref{eq:multisegment-parabolic-induction}; see also the multisegment
description of affine type-\(A\) Hecke modules \cite{Vazirani2001}.

Let \(W^P\) be the minimal right-coset representatives for \(W/W_P\).

A permutation $w\in S_N$ is called an
$({\ell_1},\ldots,{\ell_t})$-shuffle if it preserves the relative order
inside each consecutive block
\[
 I_j=\{\ell_1+\cdots+\ell_{j-1}+1,\ldots,
       \ell_1+\cdots+\ell_j\};
\]
equivalently,
\[
 a<b,\quad a,b\in I_j
 \quad\Longrightarrow\quad
 w(a)<w(b).
\]
These permutations are exactly the minimal right-coset representatives
$W^P$ for $W/W_P$.  The vectors
\begin{equation}\label{eq:shuffle-basis}
 m_w=T_w\otimes1,
 \qquad
 w\in W^P,
\end{equation}
form the natural finite shuffle basis of \(M(\mathfrak m)\).

\subsection{The finite shuffle matrices}

We use the following Bernstein presentation.  For the type-$A_{N-1}$ affine
Hecke algebra $\mathcal H_N$, let $T_1,\ldots,T_{N-1}$ be the finite Hecke
generators and let $X_1^{\pm1},\ldots,X_N^{\pm1}$ be commuting Bernstein
translation generators.

\begin{definition}[Bernstein translation monomials and Bernstein monomials]
\label{def:Bernstein-translation-monomial}
Identify the coweight lattice of the diagonal torus with
\[
 X_*(T)\simeq\mathbb Z^N.
\]
For a coweight
\[
 \eta=(\eta_1,\ldots,\eta_N)\in X_*(T),
\]
the element
\begin{equation}\label{eq:Bernstein-translation-monomial}
 X^\eta
 :=
 X_1^{\eta_1}\cdots X_N^{\eta_N}
\end{equation}
in the commutative Bernstein torus subalgebra
\[
 \mathcal A
 =
 E[X_1^{\pm1},\ldots,X_N^{\pm1}]
 \subseteq \mathcal H_N
\]
is called the \emph{Bernstein translation monomial} of exponent \(\eta\).
In root-theoretic notation we write
\[
 \Theta_\eta:=X^\eta
\]
and call the same element the \emph{Bernstein monomial} attached to
\(\eta\).  Thus, with the normalization used in this paper, ``Bernstein
translation monomial'' and ``Bernstein monomial'' refer to the same element
of \(\mathcal A\); the first terminology emphasizes its interpretation as a
translation, while the second is the root-theoretic notation used in the
Bernstein presentation.  In particular,
\[
 \Theta_\eta\Theta_{\eta'}
 =
 \Theta_{\eta+\eta'},
 \qquad
 \Theta_0=1,
 \qquad
 \Theta_\eta^{-1}=\Theta_{-\eta}.
\]
For a product of affine-Hecke factors, the definition is applied factor by
factor to the corresponding tuple of coweights.
\end{definition}

They satisfy
\[
 (T_i+1)(T_i-q_\rho)=0,
\]
the usual braid relations,
\[
 X_jX_k=X_kX_j,
 \qquad
 T_iX_j=X_jT_i\quad(j\neq i,i+1),
\]
and
\begin{equation}\label{eq:Bernstein-type-A-cross-relation}
 T_iX_iT_i=q_\rho X_{i+1}.
\end{equation}
Equivalently, in root-theoretic notation, for a simple reflection $s=s_\alpha$
and a Bernstein monomial $\Theta_\eta$ one has
\begin{equation}\label{eq:Bernstein-cross-relation-root-form}
 T_s\Theta_\eta-\Theta_{s\eta}T_s
 =(q_\rho-1)
 \frac{\Theta_\eta-\Theta_{s\eta}}
 {1-\Theta_{-\alpha^\vee}}.
\end{equation}
The quotient on the right is a finite Laurent polynomial whenever
$\eta-s\eta$ is an integral multiple of $\alpha^\vee$.  The elements
$\{T_w\Theta_\eta:w\in S_N,\eta\in\mathbb Z^N\}$ form an $E$-basis of
$\mathcal H_N$.  This is the Bernstein presentation used below; explicit
conversion to the Iwahori--Matsumoto basis is described in
\cite{HainesPettet2002}.

Fix the Bernstein presentation of the affine Hecke algebra, with translation
elements \(\Theta_\eta\).  For \(h\in\mathcal H_N\) and \(w\in W^P\), write
the parabolic Bernstein normal form
\begin{equation}\label{eq:parabolic-Bernstein-normal-form}
 T_wh
 =
 \sum_{\substack{w'\in W^P\\u\in W_P\\\eta}}
 c_{w',u,\eta}(w,h)\,
 T_{w'}\Theta_\eta T_u.
\end{equation}
The coefficients \(c_{w',u,\eta}(w,h)\) are explicit affine-Hecke structure
constants.  Formulas relating the Bernstein and Iwahori--Matsumoto
presentations may be used to compute them recursively
\cite{HainesPettet2002}.

Since the internal finite-Hecke generators act by the sign character on each
segment,
\[
 \chi_{\Delta_i}(T_s)=-1
\]
for the simple reflections internal to that segment.  Therefore the matrix
of \(h\) in the shuffle basis is
\begin{equation}\label{eq:finite-shuffle-matrix}
 C^{\mathfrak m}_{w',w}(h)
 =
 \sum_{u,\eta}
 c_{w',u,\eta}(w,h)\,
 \chi_{\mathfrak m}(\Theta_\eta)
 (-1)^{\ell(u)},
\end{equation}
where
\[
 \chi_{\mathfrak m}
 =
 \chi_{\Delta_1}\boxtimes\cdots\boxtimes\chi_{\Delta_t}
\]
on the parabolic Bernstein subalgebra.

Thus $C^{\mathfrak m}(h)$ is a completely finite matrix.  More
precisely, there is a useful block matrix before applying the segment
characters.  Write the finite free right $\mathcal H_P$-module as
\[
 \mathcal H_N
 =\bigoplus_{w\in W^P}T_w\mathcal H_P.
\]
For $h\in\mathcal H_N$, right multiplication followed by Bernstein--Bruhat
straightening has a block matrix
\begin{equation}\label{eq:pre-specialization-block-matrix}
 \widetilde C(h)
 =\bigl(B_{w',w}(h)\bigr)_{w',w\in W^P},
 \qquad
 B_{w',w}(h)\in\mathcal H_P,
\end{equation}
characterized by
\[
 T_wh=\sum_{w'\in W^P}T_{w'}B_{w',w}(h).
\]
The \emph{outer} block indices $w,w'$ record the inter-segment shuffle
reduction.  Each \emph{inner block}
$B_{w',w}(h)$ lies in the tensor product
$\mathcal H_P=\bigotimes_i\mathcal H_{\ell_i}$ and therefore contains the
internal reductions inside the individual segments.  Applying the external
segment character entrywise gives the scalar shuffle matrix
\begin{equation}\label{eq:finite-shuffle-block-specialization}
 C^{\mathfrak m}(h)
 =\bigl(\chi_{\mathfrak m}(B_{w',w}(h))\bigr)_{w',w\in W^P}.
\end{equation}

\subsection{The universal coset basis and the global reduction matrix}

Choose a \(J\)-stable lattice
\[
 \lambda^\circ\subseteq\lambda
\]
and an \(\Ocoeff\)-basis
\[
 u_1,\ldots,u_r.
\]
Choose representatives \(\mathscr R\) for \(J\backslash G\).  For
\(x\in\mathscr R\) and \(1\leq a\leq r\), let
\[
 \phi_{x,a}\in\mathcal P_\lambda
\]
be supported on \(Jx\) and satisfy
\[
 \phi_{x,a}(x)=u_a.
\]
Then
\[
 \{\phi_{x,a}\}_{x,a}
\]
is a Hamel basis of \(\mathcal P_\lambda\).
Indeed, an element $f\in\cInd_J^G\lambda$ is a function
$f:G\to\lambda$ satisfying $f(jg)=\lambda(j)f(g)$ and having compact
support modulo $J$.  Since $J$ is open, compact support modulo $J$ meets
only finitely many left cosets $Jx$ with $x\in\mathscr R$.  On such a coset,
$f$ is determined uniquely by its value
\[
 f(x)=\sum_{a=1}^r c_{x,a}u_a.
\]
The equivariance relation then gives
\[
 f|_{Jx}=\sum_{a=1}^r c_{x,a}\phi_{x,a}|_{Jx}.
\]
Summing over the finitely many cosets meeting the support yields the finite
expansion
\[
 f=\sum_{x,a}c_{x,a}\phi_{x,a},
\]
so the family spans.  If a finite linear combination
$\sum_{x,a}c_{x,a}\phi_{x,a}$ is zero, restriction to one coset $Jx$ and
evaluation at $x$ gives $\sum_a c_{x,a}u_a=0$; since the $u_a$ form a basis,
all $c_{x,a}$ vanish.  Thus the family is linearly independent and hence a
Hamel basis.

The natural spanning vectors of \(\pi_{\mathrm{gen}}(r)\) are
\begin{equation}\label{eq:q-xaw}
 q_{x,a,w}
 =
 \phi_{x,a}\otimes m_w,
 \qquad
 x\in\mathscr R,\quad
 1\leq a\leq r,\quad
 w\in W^P.
\end{equation}
They are usually not independent.

For \(h\in\mathcal H\), write
\begin{equation}\label{eq:Hecke-on-universal-basis}
 \phi_{x,a}h
 =
 \sum_{y,b}
 H^h_{(y,b),(x,a)}\phi_{y,b}.
\end{equation}
The tensor relation
\[
 (\phi h)\otimes m
 =
 \phi\otimes(hm)
\]
and \eqref{eq:finite-shuffle-matrix} give the explicit relations
\begin{equation}\label{eq:master-reduction-relation}
 \sum_{y,b}
 H^h_{(y,b),(x,a)}
 q_{y,b,w}
 =
 \sum_{w'\in W^P}
 C^{\mathfrak m}_{w',w}(h)
 q_{x,a,w'}.
\end{equation}

Let
\[
 \Xi=\mathscr R\times\{1,\ldots,r\}\times W^P
\]
be the set of \emph{spanning indices}; thus
$\xi=(x,a,w)\in\Xi$ labels the spanning vector $q_\xi=q_{x,a,w}$.

\begin{definition}[Affine--Bruhat shelling]
\label{def:affine-Bruhat-shelling}
Choose an Iwahori subgroup $I\subseteq G$ and a proper length function
$\ell_*:\widetilde W\to\mathbb Z_{\geq0}$ on the extended affine Weyl group,
obtained from the affine Coxeter length together with any proper norm on the
central translation lattice.  For $x\in\mathscr R$, define
\[
 h(x)=\min\{\ell_*(w):Jx\cap IwI\neq\varnothing\}.
\]
An \emph{affine--Bruhat shelling} is the resulting filtration of spanning
indices by the integer $h(x)$, together with a fixed total ordering inside
each shell.  Because $\ell_*$ is proper and the residue field is finite, each
shell contains only finitely many $J$-cosets.  We write $h(\xi)=h(x)$ for
$\xi=(x,a,w)$.
\end{definition}

\begin{definition}[Deterministic pivot rule]
\label{def:deterministic-pivot-rule}
Order $\Xi$ first by shell and then by the chosen order inside the shell.
A \emph{deterministic pivot rule} is a row-echelon choice of a basis of the
relation space generated by \eqref{eq:master-reduction-relation} such that:
\begin{enumerate}
 \item every chosen relation has a unique pivot index $\xi$ with nonzero
 coefficient;
 \item after solving for $q_\xi$, only indices of strictly smaller order
 occur on the right;
 \item no pivot index occurs as the pivot of two chosen relations.
\end{enumerate}
The rule is deterministic once the ordering of relations and the usual
row-reduction tie-breaking convention are fixed.  Indices which never occur
as pivots are called \emph{nonpivot indices}.
\end{definition}

\begin{lemma}[Triangular reduction]
\label{lem:triangular-reduction-normal-form}
Assume the chosen row-echelon relations span the full kernel of the map from
the free span of the $q_\xi$ to
$\mathcal P_\lambda\otimes_{\mathcal H}M(\mathfrak m)$.  Then repeated pivot
reduction terminates, every spanning vector has a unique expression in the
nonpivot vectors, and the nonpivot vectors form a Hamel basis of the
quotient.
\end{lemma}

\begin{proof}
Every reduction strictly decreases the well-founded shell/order complexity,
so an infinite reduction chain is impossible.  The spanning property gives
existence of a normal form.  If two normal forms represented the same vector,
their difference would be a nonzero linear combination of nonpivot vectors
lying in the relation space.  This contradicts the row-echelon property,
since every nonzero vector in the relation space has a largest pivot index.
Hence the normal form is unique and the nonpivot classes are a basis.
\end{proof}

Choose such an affine--Bruhat shelling and deterministic pivot rule for the
relations \eqref{eq:master-reduction-relation}.  Let
\[
 \mathscr N
\]
be the nonpivot indices.  We take
\[
 e_\mu=q_\mu,\qquad\mu\in\mathscr N,
\]
as our type--Bruhat Hamel basis.  Every spanning vector has a unique
normal-form expansion
\begin{equation}\label{eq:global-reduction-matrix}
 q_\xi
 =
 \sum_{\mu\in\mathscr N}
 R_{\mu;\xi}e_\mu.
\end{equation}
The matrix \(R\) is the global reduction matrix.

It is computed recursively from the finite matrices
\(H^h\) and \(C^{\mathfrak m}(h)\) appearing in
\eqref{eq:master-reduction-relation}.

\subsection{Path formula for the global reduction matrix}

We now formalize the path expansion implicit in the deterministic reduction
procedure.

\begin{definition}[Reduction complexity and reduction graph]
\label{def:global-reduction-graph}
For a spanning index \(\xi=(x,a,w)\in\Xi\), define its
\emph{reduction complexity} by
\[
 \operatorname{comp}(\xi)
 =
 \bigl(h(\xi),p(\xi)\bigr),
\]
ordered lexicographically, where \(h(\xi)\) is the shell number from
Definition~\ref{def:affine-Bruhat-shelling} and \(p(\xi)\) is the position of
\(\xi\) in the fixed deterministic ordering inside that shell.

Whenever a chosen row-echelon relation
\eqref{eq:master-reduction-relation} is solved for its pivot \(q_\xi\), draw a
directed edge
\[
 e:\xi\longrightarrow\xi'
\]
for every lower-complexity term \(q_{\xi'}\) occurring on the right-hand
side.  The resulting directed graph on \(\Xi\) is called the
\emph{global reduction graph} and is denoted by \(\Gamma_{\mathrm{red}}\).
Its terminal vertices are the nonpivot indices \(\mathscr N\).
\end{definition}

\begin{lemma}[Acyclicity and termination]
\label{lem:global-reduction-graph-acyclic}
The graph \(\Gamma_{\mathrm{red}}\) is acyclic.  Every directed path is
finite and terminates at a vertex of \(\mathscr N\).  Consequently repeated
pivot reduction terminates, and the terminal normal form is independent of
all reduction choices.
\end{lemma}

\begin{proof}
By Definition~\ref{def:deterministic-pivot-rule}, every term occurring after
a pivot is solved has strictly smaller shell/order complexity.  Hence
\(\operatorname{comp}(\xi')<\operatorname{comp}(\xi)\) along every directed
edge \(\xi\to\xi'\), so a directed cycle is impossible and every directed
path is finite.  A terminal vertex is precisely an index which is never
chosen as a pivot, hence belongs to \(\mathscr N\).  Independence of the
terminal expression follows from the uniqueness of normal form in
Lemma~\ref{lem:triangular-reduction-normal-form}.
\end{proof}

\begin{definition}[Edge data and reduction paths]
\label{def:global-reduction-path}
Let \(e:\xi\to\xi'\) be an edge of
\(\Gamma_{\mathrm{red}}\).  After solving the pivot relation, write the
coefficient of \(q_{\xi'}\) as the product of
\begin{enumerate}
 \item a pure Bernstein translation \(\Theta_{\eta(e)}\);
 \item a scalar factor made from signs and nonnegative powers of the affine
 Hecke parameters \(q_\rho\) and \(q_\rho-1\);
 \item a type-intertwining matrix \(U(e)\).
\end{enumerate}
The translation exponent \(\eta(e)\) lies in the coweight lattice of the
relevant affine-Hecke factor.  After choosing an integral normalization of
the type, we require
\[
 U(e)\in M_r(\Ocoeff).
\]

A \emph{reduction path} is a directed path
\[
 \gamma:
 \xi=\xi_0\longrightarrow\xi_1\longrightarrow\cdots
 \longrightarrow\xi_s=\mu,
 \qquad \mu\in\mathscr N.
\]
If \(e_k:\xi_{k-1}\to\xi_k\) is its \(k\)-th edge, define
\begin{equation}\label{eq:path-C-and-eta}
 C(\gamma)=c(e_s)\cdots c(e_1),
 \qquad
 \eta(\gamma)=\sum_{k=1}^s\eta(e_k),
\end{equation}
where \(c(e_k)\) is the coefficient with the pure Bernstein translation
removed.  At block level the product is ordered matrix multiplication; a
scalar coefficient is obtained by taking the relevant matrix entry at the
end.
\end{definition}

\begin{proposition}[Finite path formula for the global reduction matrix]
\label{prop:global-reduction-path-formula}
For every spanning index \(\xi\in\Xi\) and every nonpivot
\(\mu\in\mathscr N\), only finitely many reduction paths
\(\gamma:\xi\leadsto\mu\) occur, and
\begin{equation}\label{eq:R-gallery-formula}
 R_{\mu;\xi}
 =
 \sum_{\gamma:\xi\leadsto\mu}
 C(\gamma)\,
 \chi_{\mathfrak m}
 \bigl(\Theta_{\eta(\gamma)}\bigr).
\end{equation}
After separating the integral type matrices and the affine-Hecke scalars,
this can be written
\begin{equation}\label{eq:R-gallery-expanded}
 R_{\mu;\xi}
 =
 \sum_{\gamma:\xi\leadsto\mu}
 \epsilon(\gamma)\,
 \prod_\rho q_\rho^{c_\rho(\gamma)}
 (q_\rho-1)^{f_\rho(\gamma)}
 \chi_{\mathfrak m}
 \bigl(\Theta_{\eta(\gamma)}\bigr)
 U(\gamma),
\end{equation}
where
\[
 \epsilon(\gamma)\in\{\pm1\},
 \qquad
 c_\rho(\gamma),f_\rho(\gamma)\in\mathbb Z_{\geq0},
 \qquad
 U(\gamma)\in M_r(\Ocoeff).
\]
\end{proposition}

\begin{proof}
Lemma~\ref{lem:global-reduction-graph-acyclic} implies that every reduction
tree rooted at a fixed \(\xi\) has finite depth.  Each pivot relation has
finite support, hence the tree has only finitely many branches.  Expanding
recursively and collecting the contribution of the branches terminating at
\(\mu\) gives \eqref{eq:R-gallery-formula}.  Multiplicativity of Bernstein
translations gives the accumulated exponent \(\eta(\gamma)\), while the
remaining scalar and type factors multiply in the order of the path.  This
gives \eqref{eq:R-gallery-expanded}.
\end{proof}

For a matrix \(U=(u_{ab})\in M_r(E)\) we use the matrix valuation
\begin{equation}\label{eq:matrix-valuation-definition}
 v_E(U)=\min_{a,b}v_E(u_{ab}),
 \qquad v_E(0)=+\infty.
\end{equation}
Thus \(U\in M_r(\Ocoeff)\) implies \(v_E(U)\geq0\), and
\(v_E(UV)\geq v_E(U)+v_E(V)\).

\begin{corollary}[Pathwise valuation lower bound]
\label{cor:global-reduction-path-valuation}
For every \(\mu\in\mathscr N\) and \(\xi\in\Xi\),
\begin{equation}\label{eq:R-valuation-lower-bound}
 v_E(R_{\mu;\xi})
 \geq
 \min_{\gamma:\xi\leadsto\mu}
 \left[
 \sum_\rho c_\rho(\gamma)v_E(q_\rho)
 +
 v_E\!\left(
 \chi_{\mathfrak m}(\Theta_{\eta(\gamma)})
 \right)
 +
 v_E(U(\gamma))
 \right].
\end{equation}
The factors \(q_\rho-1\) do not occur on the right because
\(q_\rho-1\in\Ocoeff^\times\).  Cancellation among distinct paths can only
increase the valuation.
\end{corollary}

\begin{proof}
Apply the ultrametric inequality to
\eqref{eq:R-gallery-expanded}.  The sign factors and the powers of
\(q_\rho-1\) have valuation zero, while the remaining factors contribute
additively to the displayed lower bound.
\end{proof}

\subsection{Exact matrix coefficients of an individual group element}

The point of the construction is that the action of an individual
\(g\in G\) is block-monomial on the universal coset basis.

Recall that $\mathscr R$ is the fixed set of representatives for the left cosets $J\backslash G$ chosen above.  For $x\in\mathscr R$, write
\[
 xg^{-1}=j_{x,g}x_g,
 \qquad
 j_{x,g}\in J,
 \quad
 x_g\in\mathscr R.
\]
Then
\begin{equation}\label{eq:g-on-universal-coset-basis}
 g\phi_{x,a}
 =
 \sum_{b=1}^r
 \lambda(j_{x,g}^{-1})_{ba}\,
 \phi_{x_g,b}.
\end{equation}

Let
\[
 e_\nu=q_{x,a,w}
\]
be a normal-form basis vector.  Combining
\eqref{eq:g-on-universal-coset-basis} with
\eqref{eq:global-reduction-matrix} gives the exact formula
\begin{equation}\label{eq:exact-smooth-matrix-coefficient}
 \boxed{
 A^{\mathrm{sm}}_{\mu\nu}(g)
 =
 \sum_{b=1}^r
 R_{\mu;(x_g,b,w)}
 \lambda(j_{x,g}^{-1})_{ba}.
 }
\end{equation}

Now choose a torus weight basis
\[
 \{v_\beta\}
\]
of \(\pi_{\mathrm{alg}}(r)\).  For
\[
 a_j=\operatorname{diag}(\varpi_FI_j,I_{n-j}),
\]
write
\[
 a_jv_\beta=\eta_\beta(a_j)v_\beta.
\]
The Hamel basis of \(V\) is
\[
 e_{\nu,\beta}=e_\nu\otimes v_\beta.
\]
For an arbitrary $g\in G$, define the locally algebraic matrix coefficients by
\begin{equation}\label{eq:definition-locally-algebraic-A}
 g e_{\nu,\beta}
 =\sum_{\mu,\gamma}
 A_{(\mu,\gamma),(\nu,\beta)}(g)e_{\mu,\gamma}.
\end{equation}
Thus the first pair of indices records the output smooth and algebraic basis
vectors and the second pair records the input vectors.  For a Cartan element
$a_j$, the chosen algebraic weight basis is diagonal, so the algebraic index
is preserved.  Hence
\begin{equation}\label{eq:exact-locally-algebraic-matrix-coefficient}
 \boxed{
 A_{(\mu,\gamma),(\nu,\beta)}(a_j)
 =
 \delta_{\gamma\beta}\,
 \eta_\beta(a_j)
 \sum_{b=1}^r
 R_{\mu;(x_j,b,w)}
 \lambda(j_{x,j}^{-1})_{ba},
 }
\end{equation}
where
\[
 xa_j^{-1}=j_{x,j}x_j.
\]

Consequently 
\begin{align}
 v_E
 \left(
 A_{(\mu,\beta),(\nu,\beta)}(a_j)
 \right)
 \geq\;&
 v_E(\eta_\beta(a_j))
 \nonumber\\
 &+
 \min_b
 \left\{
 v_E(R_{\mu;(x_j,b,w)})
 +
 v_E(\lambda(j_{x,j}^{-1})_{ba})
 \right\}.
 \label{eq:matrix-coefficient-valuation-basic}
\end{align}
If \(\lambda^\circ\) is \(J\)-stable, then
\[
 v_E(\lambda(j)_{ba})\geq0.
\]

\subsection{Defects in the multisegment reduction}

We now organize the multisegment data and the resulting valuation calculation
in definition--lemma--proposition form.

\begin{definition}[Atomic Weil--Deligne blocks]
\label{def:atomic-WD-blocks}
Recall the indecomposable decomposition
\eqref{eq:WD-indecomposable-decomposition}:
\[
 D=\bigoplus_{i=1}^tD_i,
 \qquad
 D_i=D(\tau_i,\ell_i).
\]
Write each indecomposable block as its monodromy chain
\begin{equation}\label{eq:atomic-WD-chain}
 D_i
 =D_{i,0}\oplus D_{i,1}\oplus\cdots\oplus D_{i,\ell_i-1},
 \qquad
 D_{i,j}=D_{i,0}(j),
\end{equation}
where \(N(D_{i,j})\subseteq D_{i,j-1}\) for \(j>0\).  The individual
summands \(D_{i,j}\) are called the \emph{atomic Weil--Deligne blocks}, or
simply the \emph{atomic blocks}.  Thus an atomic block is one irreducible
Weil constituent together with one fixed twist, before adjoining it to the
rest of its monodromy chain.  Enumerate all atomic blocks as
\[
 D'_1,\ldots,D'_B,
 \qquad
 B=\sum_{i=1}^t\ell_i.
\]
The rank of \(D'_a\) is the dimension of its underlying irreducible Weil
representation; these ranks need not be equal when several cuspidal lines
occur.  The Newton number \(t_N(D')\) denotes the additive Newton degree of
the corresponding Frobenius module, with the same normalization as in Hu's
admissibility criterion \cite{MR2560407}.
\end{definition}

\begin{definition}[Admissible Weil--Deligne prefix and admissible ordering]
\label{def:admissible-wd-prefix}
Let \(\nu\in S_B\).  For \(1\leq k\leq B\), set
\[
 D_{\nu,\leq k}
 =
 \bigoplus_{a=1}^kD'_{\nu(a)}.
\]
We call \(D_{\nu,\leq k}\) an \emph{admissible Weil--Deligne prefix} if it
is a Weil--Deligne subobject of \(D\).  In the atomic decomposition
\eqref{eq:atomic-WD-chain}, this means precisely that the chosen set of
atomic blocks is closed under all inherited monodromy arrows: whenever
\(D_{i,j}\) with \(j>0\) belongs to the prefix, every predecessor required by
the monodromy chain, in particular \(D_{i,j-1}\), also belongs to the
prefix.  Equivalently, the set of selected positions in each monodromy chain
is downward closed.

An ordering \(\nu\in S_B\) is \emph{admissible for prefixes} if every prefix
that is used in the reduction is admissible in this sense.  A
\emph{proper admissible prefix} is one of rank strictly smaller than
\(\operatorname{rk}D\).
\end{definition}

\begin{definition}[Prefix ranks and defects]
\label{def:prefix-defects}
For an ordering \(\nu\) admissible for prefixes, define prefix rank as
\[
 r_{\nu,k}
 =
 \sum_{a=1}^k\operatorname{rk}D'_{\nu(a)}
\]
and define defects as
\begin{equation}\label{eq:defects-multisegment}
 \Delta_{\nu,k}
 =
 \sum_{a=1}^k t_N(D'_{\nu(a)})
 -
 [E:F]
 \sum_{\ell=1}^{r_{\nu,k}}\sum_\sigma i_{\ell,\sigma}.
\end{equation}
We also set
\begin{equation}\label{eq:kappa-multisegment}
 \kappa_{E,F}
 =
 \frac{e(E/\mathbb Q_p)}
 {e(F/\mathbb Q_p)[E:F]}.
\end{equation}
\end{definition}

\begin{lemma}[Inequalities on admissible prefixes]
\label{lem:admissible-prefix-inequalities}
Assume that the filtered Weil--Deligne module attached to \(r\) is
admissible.  Then every proper admissible Weil--Deligne prefix occurring in
the multisegment reduction satisfies
\[
 \Delta_{\nu,k}\geq0,
\]
and the full-rank prefix satisfies equality.
\end{lemma}

\begin{proof}
By Definition~\ref{def:admissible-wd-prefix},
\(D_{\nu,\leq k}\) is a Weil--Deligne subobject.  Hu's admissibility
criterion therefore applies to this subobject and gives that its Newton
degree is at least the Hodge degree determined by the first
\(r_{\nu,k}\) Hodge jumps.  In the normalization of
Definition~\ref{def:prefix-defects}, this is exactly
\(\Delta_{\nu,k}\geq0\).  For the full module, weak admissibility gives
equality of total Newton and Hodge degrees.
\end{proof}

\begin{definition}[Positively oriented colored affine gallery]
\label{def:colored-affine-gallery}
A \emph{colored affine gallery} is a positively oriented affine--Bruhat
reduction path
\[
 \gamma:\xi_0\longrightarrow\xi_1\longrightarrow\cdots
 \longrightarrow\xi_s
\]
in the reduction graph underlying \eqref{eq:R-gallery-formula}, together
with labels, called \emph{colors}, on the distinguished positive
prefix-wall crossings.  The colors are the atomic blocks
\(D'_1,\ldots,D'_B\).  Crossing a prefix wall with color \(D'_a\) means that
the growing atomic prefix is enlarged by adjoining \(D'_a\).  We record a
color only when that atomic block enters the growing prefix for the first
time; repeated elementary reduction steps carrying the same block do not
create a new position in the ordering.

The colored gallery is \emph{complete} if every atomic block occurs, and it
is \emph{admissibly colored} if, after each new color is recorded, the direct
sum of the atomic blocks recorded so far is an admissible Weil--Deligne
prefix in the sense of Definition~\ref{def:admissible-wd-prefix}.  Reading
the first occurrences of the colors along a complete admissibly colored
gallery gives an ordering
\[
 \nu_\gamma=(\nu_\gamma(1),\ldots,\nu_\gamma(B))\in S_B.
\]
The gallery is \emph{positively oriented} when its translation increments
lie in the chosen dominant cone.  Accordingly, the coefficients of the
noncentral fundamental coweights in its accumulated translation are
nonnegative.
\end{definition}

\begin{lemma}[Prefix decomposition of the gallery translation]
\label{lem:gallery-prefix-translation-decomposition}
Let \(\gamma\) be a complete admissibly colored positively oriented gallery,
and let \(\nu_\gamma\) be the ordering obtained from its colors.  Let
\(T\subset\GL_N\) be the diagonal torus and let
\(\varepsilon_i^\vee\in X_*(T)\) be the \(i\)-th coordinate coweight.  For
\(1\leq j\leq N\), put
\begin{equation}\label{eq:standard-fundamental-coweights}
 \omega_j^\vee
 =\varepsilon_1^\vee+\cdots+\varepsilon_j^\vee,
 \qquad
 \omega_j^\vee(t)
 =\operatorname{diag}(\underbrace{t,\ldots,t}_{j},1,\ldots,1).
\end{equation}
Then \(\omega_1^\vee,\ldots,\omega_{N-1}^\vee\) are the standard
fundamental coweights for the upper-triangular Borel of \(\GL_N\), whereas
\(\omega_N^\vee\) is central.  The accumulated translation of \(\gamma\)
has, uniquely modulo the centre, the form
\begin{equation}\label{eq:gallery-translation-prefix-decomposition}
 \eta(\gamma)
 =z_\gamma+
 \sum_km_{\gamma,k}\,
 \omega_{r_{\nu_\gamma,k}}^\vee,
 \qquad
 m_{\gamma,k}\in\mathbb Z_{\geq0},
\end{equation}
where
\(z_\gamma\in X_*(Z(\GL_N))=\mathbb Z\omega_N^\vee\), and the sum runs over
the proper admissible prefixes selected by the gallery.
\end{lemma}

\begin{proof}
For a dominant coweight
\(\lambda=(\lambda_1,\ldots,\lambda_N)\) with
\(\lambda_1\geq\cdots\geq\lambda_N\), one has
\[
 \lambda
 =
 \lambda_N\omega_N^\vee
 +
 \sum_{j=1}^{N-1}
 (\lambda_j-\lambda_{j+1})\omega_j^\vee.
\]
The coefficients of the noncentral fundamental coweights are therefore
nonnegative.  By the coloring convention of
Definition~\ref{def:colored-affine-gallery}, a distinguished prefix-wall
crossing can occur only at a rank \(j=r_{\nu_\gamma,k}\) of an admissible
prefix.  Collecting equal fundamental-coweight contributions along the
gallery gives \eqref{eq:gallery-translation-prefix-decomposition}; the
remaining multiple of \(\omega_N^\vee\) is the central term \(z_\gamma\).
Uniqueness follows from the displayed fundamental-coweight expansion modulo
\(\mathbb Z\omega_N^\vee\).
\end{proof}

\begin{lemma}[Valuation of a fundamental prefix translation]
\label{lem:prefix-valuation}
For every proper admissible prefix \(D_{\nu,\leq k}\):
\begin{align}
 &v_E\!\left(
   \chi_{\mathfrak m}(\Theta_{\omega_{r_{\nu,k}}^\vee})
   \eta_{\mathrm{alg}}(\omega_{r_{\nu,k}}^\vee)
 \right) \notag\\
 &\qquad=
 \kappa_{E,F}
 \left(
  t_N(D_{\nu,\leq k})
  -[E:F]
   \sum_{\ell=1}^{r_{\nu,k}}\sum_\sigma i_{\ell,\sigma}
 \right)
 =\kappa_{E,F}\Delta_{\nu,k}.
 \label{eq:prefix-valuation}
\end{align}
\end{lemma}

\begin{proof}
In the Bernstein normalization used for the type, the smooth Bernstein
character on the translation
\(\omega_{r_{\nu,k}}^\vee\) records the Newton degree of the
Weil--Deligne subobject \(D_{\nu,\leq k}\).  The algebraic highest-weight
character contributes the negative Hodge degree of rank
\(r_{\nu,k}\).  Taking \(v_E\), and converting the valuation normalization
by \(\kappa_{E,F}\), gives the displayed identity.  This is precisely Hu's
prefix inequality written in the present Bernstein normalization.
\end{proof}

\begin{proposition}[Valuation of a colored gallery]
\label{prop:gallery-valuation}
Let \(\gamma\) be a complete admissibly colored positively oriented gallery.
Then
\begin{equation}\label{eq:gallery-valuation}
 v_E\left(
 \chi_{\mathfrak m}(\Theta_{\eta(\gamma)})
 \eta_{\mathrm{alg}}(\eta(\gamma))
 \right)
 =
 \kappa_{E,F}
 \sum_k
 m_{\gamma,k}
 \Delta_{\nu_\gamma,k}.
\end{equation}
In particular, if the filtered module is admissible, the right-hand side is
nonnegative.
\end{proposition}

\begin{proof}
Both \(\chi_{\mathfrak m}\) and \(\eta_{\mathrm{alg}}\) are multiplicative
on the translation lattice, and \(v_E\) is additive on products.  Apply
Lemma~\ref{lem:prefix-valuation} to the decomposition
\eqref{eq:gallery-translation-prefix-decomposition}.  This gives
\begin{align*}
 &v_E\left(
 \chi_{\mathfrak m}(\Theta_{\eta(\gamma)})
 \eta_{\mathrm{alg}}(\eta(\gamma))
 \right)\\
 &\qquad=
 v_E\left(
 \chi_{\mathfrak m}(\Theta_{z_\gamma})
 \eta_{\mathrm{alg}}(z_\gamma)
 \right)
 +\kappa_{E,F}
 \sum_km_{\gamma,k}\Delta_{\nu_\gamma,k}.
\end{align*}
The first term is zero: a central coweight tests the full
Weil--Deligne module, and weak admissibility gives the total equality
\[
 t_N(D)
 =[E:F]\sum_{\ell=1}^{N}\sum_\sigma i_{\ell,\sigma}.
\]
Equivalently, the normalized smooth--algebraic central character is unitary.
The final nonnegativity follows from
Lemma~\ref{lem:admissible-prefix-inequalities}.
\end{proof}

\begin{proposition}[Multisegment Cartan coefficient estimate]
\label{prop:multisegment-Cartan-coefficient-estimate}
Combining
\eqref{eq:R-valuation-lower-bound},
\eqref{eq:exact-locally-algebraic-matrix-coefficient},
and Proposition~\ref{prop:gallery-valuation}, one obtains
\begin{equation}\label{eq:multisegment-A-valuation}
 \boxed{
 v_E(A_{\mu\nu}(a_j))
 \geq
 \min_{\gamma\in\mathcal G_j(\mu,\nu)}
 \left[
 \kappa_{E,F}
 \sum_km_{\gamma,k}\Delta_{\nu_\gamma,k}
 +
 c(\gamma)v_E(q_\rho)
 +
 v_E(U(\gamma))
 \right].
 }
\end{equation}
For several cuspidal lines, the corresponding contributions on the
right-hand side are summed over the affine-Hecke factors.
\end{proposition}

\begin{proof}
For each reduction gallery contributing to the matrix coefficient, the
smooth Bernstein translation and the algebraic weight factor combine as in
Proposition~\ref{prop:gallery-valuation}.  The remaining reduction
coefficient has valuation at least
\(c(\gamma)v_E(q_\rho)+v_E(U(\gamma))\) by
\eqref{eq:R-valuation-lower-bound}.  The ultrametric inequality gives the
minimum over all contributing galleries.  Several galleries may have the
same minimal valuation, but cancellation among them can only increase the
valuation, so the result is a lower bound.
\end{proof}

Formula \eqref{eq:multisegment-A-valuation} is therefore the useful
valuation direction for constructing a potential.

Fix a fundamental Cartan generator $a_j$.  A \emph{Cartan transition
edge} from $\nu$ to $\mu$ is a specified nonzero summand occurring in the
exact expansion of $a_je_\nu$ in the normal-form basis, together with one
reduction gallery $\gamma$ which produces that summand in
\eqref{eq:exact-locally-algebraic-matrix-coefficient}.  A \emph{Cartan
transition path}
\[
 p=(e_1,\ldots,e_m):\nu_0\to\nu_1\to\cdots\to\nu_m
\]
is a concatenation of such edges, all belonging to the same chosen Cartan
direction (or to a specified sequence of Cartan directions).

We make the data attached to an edge explicit.  Recall from
\eqref{eq:path-C-and-eta} that the global reduction matrix is indexed by
finite directed paths in the acyclic reduction graph.  If
$e:\nu\to\mu$ is a Cartan transition edge, the chosen \emph{reduction
gallery representing $e$} is precisely one such path
\begin{equation}\label{eq:edge-reduction-gallery-definition}
 \gamma_e:\xi_0\longrightarrow\xi_1\longrightarrow\cdots
 \longrightarrow\xi_s=\mu,
\end{equation}
where $\xi_0$ is a spanning index occurring in the universal-coset expansion
of $a_je_\nu$ and the terminal vertex is the normal-form basis index $\mu$.
The product of the edge coefficients along $\gamma_e$, together with the
Bernstein character of its accumulated translation, is the specified nonzero
summand which defines the transition edge $e$.  Thus choosing the edge means
choosing not merely the pair $(\nu,\mu)$, but also one concrete reduction
branch $\gamma_e$ producing that summand.

If $f$ is an elementary edge of the reduction graph, let $\eta(f)\in X_*(T)$
be the Bernstein translation exponent carried by that reduction step, as in
\eqref{eq:path-C-and-eta}.  The \emph{translation content} of the Cartan edge
$e$ (relative to the chosen gallery $\gamma_e$) is the accumulated coweight
\begin{equation}\label{eq:edge-translation-content-definition}
 \eta(e)=\eta(\gamma_e)
 =\sum_{f\in\gamma_e}\eta(f)\in X_*(T).
\end{equation}
For a product of affine-Hecke factors this is understood factor by factor, so
$\eta(e)$ is the corresponding tuple of coweights.

The \emph{noncentral translation content} is the image of $\eta(e)$ in
\[
 X_*(T)/X_*(Z(\GL_N)).
\]
For the positively oriented galleries considered here we use the unique
dominant representative modulo the central coweight $\omega_N^\vee$.  In
terms of the fundamental coweights defined in
\eqref{eq:standard-fundamental-coweights}, it has the form
\begin{equation}\label{eq:edge-noncentral-translation-content}
 \eta_{\mathrm{nc}}(e)
 =\sum_k m_{e,k}\,
   \omega_{r_{\nu_e,k}}^\vee,
 \qquad m_{e,k}\geq0,
\end{equation}
where the sum is over proper admissible prefixes.  Equivalently, if a dominant
representative is written in coordinates as
$(\lambda_1,\ldots,\lambda_N)$ with
$\lambda_1\geq\cdots\geq\lambda_N$, then modulo the centre
\[
 (\lambda_1,\ldots,\lambda_N)
 \equiv
 \sum_{j=1}^{N-1}(\lambda_j-\lambda_{j+1})\omega_j^\vee,
\]
so the coefficients of the noncentral fundamental coweights are nonnegative.
In our colored gallery only the indices \(j=r_{\nu_e,k}\) corresponding to
admissible Weil--Deligne prefixes occur.  By
Definition~\ref{def:admissible-wd-prefix}, such a prefix is a direct sum of
atomic blocks
\[
 P=\bigoplus_{a=1}^kD'_{\nu_e(a)}
\]
which is itself a Weil--Deligne subobject.  Concretely, in every monodromy
chain
\[
 D_{i,0},D_{i,1},\ldots,D_{i,\ell_i-1},
\]
the set of blocks belonging to \(P\) must be downward closed: if
\(D_{i,j}\) occurs with \(j>0\), then the predecessor
\(D_{i,j-1}\), and hence all earlier blocks required by repeated application
of \(N\), already occur in \(P\).  This is the precise admissibility
condition on a prefix.

We finally explain in detail how the gallery determines the ordering
\(\nu_e\).  The coloring is part of the data of the chosen reduction gallery
\(\gamma_e\).  At each distinguished positive prefix-wall crossing, the
current growing Weil--Deligne prefix is enlarged by one atomic block
\(D'_a\); that crossing is assigned the color \(a\).  Traverse
\(\gamma_e\) from its initial spanning index to its terminal normal-form
index and start with the empty ordered list.  Whenever a color \(a\) is
encountered for the first time, append \(a\) to the list.  If subsequent
elementary reduction steps carry the same color, they represent further
motion inside the same colored strip and do not create a second copy of that
block.  Because the colored gallery is complete, every atomic block is
eventually encountered, and the resulting first-occurrence list is a
permutation
\[
 \nu_e=(\nu_e(1),\ldots,\nu_e(B))
\]
of the atomic blocks.

After the first \(k\) distinct colors have appeared, the growing prefix is
\[
 D_{\nu_e,\leq k}
 =\bigoplus_{a=1}^kD'_{\nu_e(a)}.
\]
The requirement that the gallery be admissibly colored means that every such
intermediate sum which is used as a prefix is closed under the inherited
monodromy arrows.  Thus, when a block \(D_{i,j}\) is first added, all
predecessors needed for \(N\)-stability have already appeared among the
earlier colors.  Consequently every initial segment of the ordering selected
by the gallery is an admissible Weil--Deligne prefix, and the associated
rank \(r_{\nu_e,k}\) is exactly the rank of the subobject tested by the
corresponding inequality.  This is the precise meaning of an
``admissible ordering selected by the gallery.''

With these conventions, define the \emph{Cost} of the edge by
\begin{equation}\label{eq:edge-cost-definition}
 H(e)=\kappa_{E,F}\sum_km_{e,k}\Delta_{\nu_e,k}.
\end{equation}
If several cuspidal-line Hecke factors occur, this expression is summed over
those factors.  Define the \emph{residual integral cost} by
\begin{equation}\label{eq:residual-edge-cost-definition}
 B(e)=\sum_\rho c_\rho(e)v_E(q_\rho)+v_E(U(e)),
\end{equation}
where $c_\rho(e)\geq0$ is the power of the parameter $q_\rho$ contributed by
affine-Hecke straightening in the $\rho$-factor and $U(e)$ is the ordered
type-intertwining matrix attached to the edge; its valuation is the matrix
valuation \eqref{eq:matrix-valuation-definition}.  Thus $B(e)\geq0$ whenever
the type-intertwining matrices are integral.

For a Cartan transition path $p=(e_1,\ldots,e_m)$ put
\begin{equation}\label{eq:H-B-path-definitions}
 H(p)=\sum_{r=1}^mH(e_r),
 \qquad
 B(p)=\sum_{r=1}^mB(e_r).
\end{equation}
The valuation estimate \eqref{eq:multisegment-A-valuation} and
submultiplicativity show that the scalar contribution attached to $p$ has
valuation at least $H(p)+B(p)$.

By a \emph{potential} we mean a weight function
$x:\mathscr N\to\mathbb R$ on the normal-form Hamel basis such that, for
every Cartan transition path $p$ from $\nu$ to $\mu$,
\begin{equation}\label{eq:potential-definition}
 x_\mu-x_\nu\leq H(p)+B(p).
\end{equation}
It is enough to impose this inequality on elementary transition edges,
because summing the edge inequalities gives the path inequality.  With the
max norm having weights $c_\nu=q_E^{x_\nu}$, the inequality makes every
corresponding normalized path contribution have operator norm at most one.
On a finite transition graph, Proposition~\ref{prop:finite-negative-cycle-criterion}
shows that such a potential exists if and only if every directed cycle has
nonnegative total cost.

\subsection{Weighted Hamel norms and the Cartan estimate}

We separate here the estimates already obtained from the additional
inequalities which still have to be verified in order to deduce the existence
of a \(G\)-invariant norm.

\begin{definition}[Weighted type--Bruhat Hamel norm]
\label{def:weighted-type-Bruhat-Hamel-norm}
Let \(x:\mathscr N\to\mathbb R\) be a function and put
\[
 c_\nu=q_E^{x_\nu}.
\]
The associated weighted Hamel norm is
\[
 \alpha_x\left(\sum_{\nu\in\mathscr N}a_\nu e_\nu\right)
 =
 \max_{\nu}|a_\nu|c_\nu.
\]
For a matrix \(A(g)=(A_{\mu\nu}(g))\) in this basis,
\[
 \|g\|_{\alpha_x}
 =
 \sup_{\mu,\nu}
 |A_{\mu\nu}(g)|\frac{c_\mu}{c_\nu}.
\]
Consequently
\begin{equation}\label{eq:dinf-explicit-matrix-section8}
 d_\infty(\alpha_x,a_j^m\alpha_x)
 =
 \log
 \max\left\{
 \sup_{\mu,\nu}
 |A_{\mu\nu}(a_j^m)|\frac{c_\mu}{c_\nu},
 \sup_{\mu,\nu}
 |A_{\mu\nu}(a_j^{-m})|\frac{c_\mu}{c_\nu}
 \right\}.
\end{equation}
\end{definition}

\begin{definition}[Cartan path costs]
\label{def:Cartan-path-costs-section8}
Let \(p=e_1\cdots e_m\) be a path in the transition graph for a fundamental
Cartan generator \(a_j\).  For every edge \(e\), write its noncentral
translation content in the form
\[
 \sum_k m_{e,k}\omega_k^\vee
\]
relative to the admissible ordering \(\nu_e\) selected by the colored
reduction gallery.  Define
\begin{equation}\label{eq:path-cost}
 H_j(p)
 =
 \kappa_{E,F}
 \sum_{e\in p}\sum_km_{e,k}\Delta_{\nu_e,k}
\end{equation}
and
\begin{equation}\label{eq:residual-integral-path-cost}
 B_j(p)
 =
 \sum_{e\in p}
 \left(
 \sum_\rho c_\rho(e)v_E(q_\rho)
 +
 v_E(U(e))
 \right).
\end{equation}
Here \(c_\rho(e)\geq0\) is the power of \(q_\rho\) contributed by affine-Hecke
straightening in the \(\rho\)-factor.  If all type intertwiners are integral,
then
\[
 B_j(p)\geq0.
\]
The analogous costs for the inverse direction \(a_j^{-1}\) are denoted by
\(H_j^-(p)\) and \(B_j^-(p)\).
\end{definition}

\begin{proposition}[Established path estimate]
\label{prop:established-Cartan-path-estimate}
Assume the gallery decomposition and the valuation identity established
above.  For every matrix coefficient of \(a_j^m\), the exact path expansion
and Corollary~\ref{cor:global-reduction-path-valuation} give
\begin{equation}\label{eq:Cartan-path-coefficient-lower-bound}
 v_E\!\left(A_{\mu\nu}(a_j^m)\right)
 \geq
 \min_{p\in\mathcal P_j^{(m)}(\nu,\mu)}
 \bigl(H_j(p)+B_j(p)\bigr),
\end{equation}
and similarly
\begin{equation}\label{eq:Cartan-inverse-path-coefficient-lower-bound}
 v_E\!\left(A_{\mu\nu}(a_j^{-m})\right)
 \geq
 \min_{p\in\mathcal P_{j,-}^{(m)}(\nu,\mu)}
 \bigl(H_j^-(p)+B_j^-(p)\bigr).
\end{equation}
Therefore
\begin{equation}\label{eq:path-bound-section8}
 \log\|a_j^m\|_{\alpha_x}
 \leq
 (\log q_E)
 \sup_{p\in\mathcal P_j^{(m)}}
 \left[
 x_{\mathrm{end}(p)}
 -
 x_{\mathrm{start}(p)}
 -
 H_j(p)
 -
 B_j(p)
 \right],
\end{equation}
with the analogous estimate for \(a_j^{-m}\).
\end{proposition}

\begin{proof}
The valuation bounds are the locally algebraic form of
\eqref{eq:R-valuation-lower-bound}, after the smooth translation valuation
and the algebraic weight valuation are combined.  Since
\(|a|=q_E^{-v_E(a)}\), multiplying a coefficient by
\(c_\mu/c_\nu=q_E^{x_\mu-x_\nu}\) gives
\[
 |A_{\mu\nu}(a_j^m)|\frac{c_\mu}{c_\nu}
 \leq
 q_E^{x_\mu-x_\nu-H_j(p)-B_j(p)}
\]
for every path contributing to that coefficient.  Taking suprema gives
\eqref{eq:path-bound-section8}.
\end{proof}

\paragraph{Results already proved.}
At this stage the following implications are unconditional within the
framework fixed above:
\begin{enumerate}
 \item the type--Bruhat reduction gives exact individual matrix coefficients,
 by \eqref{eq:exact-locally-algebraic-matrix-coefficient};
 \item the path expansion gives the lower bounds
 \eqref{eq:Cartan-path-coefficient-lower-bound} and
 \eqref{eq:Cartan-inverse-path-coefficient-lower-bound};
 \item on any finite directed quotient, the system of difference
 inequalities is solvable if and only if every directed cycle has
 nonnegative total cost, by
 Proposition~\ref{prop:finite-negative-cycle-criterion};
 \item once the selected Cartan directions, the maximal compact subgroup and
 the centre are uniformly controlled, Sections~3--5 turn the resulting
 bounded orbit into a \(G\)-invariant norm.
\end{enumerate}

\paragraph{Relations still required in the type--Bruhat model.}
To complete the invariant-norm argument one must verify the following
relations.

\begin{enumerate}
 \item \emph{Forward Cartan difference inequalities.}
 For every elementary transition edge
 \(e:\nu\to\mu\) in every fundamental positive direction \(a_j\),
 \begin{equation}\label{eq:required-forward-difference-inequality}
  x_\mu-x_\nu
  \leq
  H_j(e)+B_j(e).
 \end{equation}
 Equivalently, after exponentiating,
 \[
  |A_{\mu\nu}(a_j)|\frac{c_\mu}{c_\nu}\leq1
 \]
 whenever the edge gives a single normalized transition.

 \item \emph{Inverse Cartan difference inequalities.}
 Either one proves directly that
 \begin{equation}\label{eq:required-inverse-difference-inequality}
  x_\mu-x_\nu
  \leq
  H_j^-(e)+B_j^-(e)
 \end{equation}
 for every elementary \(a_j^{-1}\)-edge, or one reduces inverse directions
 to positive directions by compact and unitary central factors.

 \item \emph{Finite-state descent.}
 If a finite type--Bruhat quotient
 \(\pi:\mathscr N\to V(\overline\Gamma)\) is used, its edge costs must satisfy
 \begin{equation}\label{eq:required-finite-state-lower-bound-section8}
  \bar w_j(\pi(\nu)\to\pi(\mu))
  \leq
  H_j(e)+B_j(e)
 \end{equation}
 and the analogous inequality for inverse directions.

 \item \emph{No negative cycles.}
 For every directed cycle \(C\) in the common finite quotient,
 \begin{equation}\label{eq:required-cycle-inequality-section8}
  \sum_{e\in C}\bar w(e)\geq0.
 \end{equation}
 By Proposition~\ref{prop:finite-negative-cycle-criterion}, this is exactly
 the compatibility condition for the existence of a potential \(x\).

 \item \emph{Compact invariance.}
 The chosen weights must be constant on the basis permutations and integral
 type transformations induced by the selected maximal compact subgroup
 \(K\); equivalently,
 \begin{equation}\label{eq:required-K-invariance-section8}
  \alpha_x(kv)=\alpha_x(v)
  \qquad(k\in K,\ v\in V).
 \end{equation}

 \item \emph{Unitary central action.}
 If \(\omega_V\) is the central character of the locally algebraic
 representation, one needs
 \begin{equation}\label{eq:required-unitary-centre-section8}
  |\omega_V(z)|=1
  \qquad(z\in Z(G)).
 \end{equation}
\end{enumerate}

\begin{proposition}[Explicit sufficient system for an invariant norm]
\label{prop:explicit-sufficient-system-invariant-norm}
Suppose that
\eqref{eq:required-forward-difference-inequality} and
\eqref{eq:required-inverse-difference-inequality} hold for a common weight
function \(x\), or that the inverse inequalities are obtained from the
forward ones by compact and unitary central symmetry.  Assume also
\eqref{eq:required-K-invariance-section8} and
\eqref{eq:required-unitary-centre-section8}.  Then
\[
 \sup_{m\geq0}d_\infty(\alpha_x,a_j^{\pm m}\alpha_x)<\infty
 \qquad(1\leq j<n).
\]
If the inequalities are nonexpanding with zero additive constant, then every
\(a_j^{\pm1}\) acts isometrically on \(\alpha_x\).  In either case the full
\(G\)-orbit of \(\alpha_x\) is bounded, and Theorem~\ref{thm:bounded-orbit}
produces a \(G\)-invariant norm.
\end{proposition}

\begin{proof}
Summing the elementary difference inequalities along a path gives
\[
 x_{\mathrm{end}(p)}-x_{\mathrm{start}(p)}
 \leq H_j(p)+B_j(p),
\]
and similarly in the inverse direction.  Proposition
\ref{prop:established-Cartan-path-estimate} therefore bounds every positive
and negative Cartan iterate.  Compact invariance and unitary central action
remove the compact and central factors in Cartan decomposition.  Hence the
orbit is bounded, and Theorem~\ref{thm:bounded-orbit} applies.
\end{proof}

The strongest sufficient situation is the termwise inequality
\[
 H_j(e)+B_j(e)\geq0,
 \qquad
 H_j^-(e)+B_j^-(e)\geq0
\]
for every elementary edge.  Then the zero potential \(x_\nu=0\) satisfies
all difference constraints.  The point of the remaining affine-Weyl
analysis is to determine precisely when this strongest situation holds and,
when it does not, to solve the weaker cycle system
\eqref{eq:required-cycle-inequality-section8}.

\subsection{Example for \texorpdfstring{$n=3$}{n=3}: a segment of length two and a singleton}

Take \(d=1\) and
\[
 \mathfrak m=\{\Delta_1,\Delta_2\},
 \qquad
 \Delta_1=[\chi,\chi\nu],
 \qquad
 \Delta_2=[\psi].
\]
Then
\[
 W=S_3,
 \qquad
 W_P=S_2,
\]
and one may choose
\[
 W^P=\{1,s_2,s_1s_2\}.
\]
Put
\[
 m_0=1\otimes1,
 \qquad
 m_1=T_2\otimes1,
 \qquad
 m_2=T_1T_2\otimes1.
\]
The internal generator \(T_1\) acts by the sign character on the length-two
segment.  With
\[
 (T_i+1)(T_i-q_\rho)=0,
\]
direct multiplication gives
\begin{equation}\label{eq:n3-T1-shuffle-matrix}
 [T_1]_{\{m_0,m_1,m_2\}}
 =
 \begin{pmatrix}
 -1&0&0\\
 0&0&q_\rho\\
 0&1&q_\rho-1
 \end{pmatrix}
\end{equation}
and
\begin{equation}\label{eq:n3-T2-shuffle-matrix}
 [T_2]_{\{m_0,m_1,m_2\}}
 =
 \begin{pmatrix}
 0&q_\rho&0\\
 1&q_\rho-1&0\\
 0&0&-1
 \end{pmatrix}.
\end{equation}
These are explicit examples of the finite matrices
\(C^{\mathfrak m}(T_i)\).

For the individual Cartan generators
\[
 a_1=\operatorname{diag}(\varpi_F,1,1),
 \qquad
 a_2=\operatorname{diag}(\varpi_F,\varpi_F,1),
\]
the full locally algebraic matrix is not
\eqref{eq:n3-T1-shuffle-matrix} or \eqref{eq:n3-T2-shuffle-matrix}.
Instead, for a normal-form basis vector
\[
 e_{\nu,\beta}
 =
 q_{x,a,w}\otimes v_\beta,
\]
the exact coefficient is
\begin{equation}\label{eq:n3-multisegment-exact-A}
 A_{(\mu,\beta),(\nu,\beta)}(a_j)
 =
 \eta_\beta(a_j)
 \sum_b
 R_{\mu;(x_j,b,w)}
 \lambda(j_{x,j}^{-1})_{ba}.
\end{equation}
The matrices \eqref{eq:n3-T1-shuffle-matrix}--\eqref{eq:n3-T2-shuffle-matrix}
enter recursively into \(R\) through
\eqref{eq:master-reduction-relation}.

\section{The one-segment case}
\label{sec:one-segment}

We now specialize to
\[
 \pi_{\mathrm{gen}}(r)=Q(\Delta),
 \qquad
 \Delta=
 [\rho\nu^a,\rho\nu^{a+1},\ldots,\rho\nu^{a+\ell-1}],
\]
where \(\rho\) is supercuspidal on \(\GL_d(F)\) and
\[
 n=d\ell.
\]
The shuffle module disappears: \(W_P=W=S_\ell\), so \(W^P=\{1\}\).
This makes the reduction matrix \(R\) substantially more explicit.

\subsection{The sign character and the Bernstein translations}

Let \(\mathcal H_\ell\) be the affine Hecke algebra of type
\(A_{\ell-1}\), with parameter \(q_\rho\), normalized by
\[
 (T_i+1)(T_i-q_\rho)=0.
\]
The segment representation corresponds to the sign/discrete-series character
\[
 \chi_\Delta(T_i)=-1.
\]

Choose Bernstein generators \(X_1,\ldots,X_\ell\) with normalization
\begin{equation}\label{eq:Bernstein-X-relation-one-segment}
 T_iX_iT_i=q_\rho X_{i+1}.
\end{equation}
Applying \(\chi_\Delta\) gives
\[
 \chi_\Delta(X_{i+1})
 =
 q_\rho^{-1}\chi_\Delta(X_i).
\]
If
\[
 z_\Delta=\chi_\Delta(X_1),
\]
then
\begin{equation}\label{eq:segment-translation-character}
 \boxed{
 \chi_\Delta(X_i)
 =
 z_\Delta q_\rho^{-(i-1)}.
 }
\end{equation}
Consequently, for
\[
 \eta=(\eta_1,\ldots,\eta_\ell),
 \qquad
 X^\eta=X_1^{\eta_1}\cdots X_\ell^{\eta_\ell},
\]
we obtain
\begin{equation}\label{eq:segment-translation-character-general}
 \boxed{
 \chi_\Delta(X^\eta)
 =
 z_\Delta^{|\eta|}
 q_\rho^{-\sum_{i=1}^{\ell}(i-1)\eta_i}.
 }
\end{equation}

\subsection{The local panel reduction matrix}

First suppose the type is scalar.  At a simple affine panel there are
\(q_\rho+1\) adjacent chambers
\[
 C_0,C_1,\ldots,C_{q_\rho}.
\]
The sign relation gives
\begin{equation}\label{eq:panel-sign-relation}
 [C_0]+\cdots+[C_{q_\rho}]=0.
\end{equation}
Choosing \(C_{q_\rho}\) as pivot,
\[
 [C_{q_\rho}]
 =
 -\sum_{i=0}^{q_\rho-1}[C_i].
\]
Thus the local reduction matrix is a $q_\rho\times(q_\rho+1)$ matrix, can be written as

\begin{equation}\label{eq:local-R-scalar}
R_s=
	\begin{pmatrix}
		1 & 0 & \cdots & 0 & -1\\
		0 & 1 & \cdots & 0  & -1\\
		\vdots & \vdots & \ddots & \vdots & \vdots \\
		0 & 0 & \cdots & 1 & -1
	\end{pmatrix}.
\end{equation}

For a type of dimension \(r\), choose an integral type basis.  Let
\[
 U_0,\ldots,U_{q_\rho-1}\in\GL_r(\Ocoeff)
\]
be the corresponding integral intertwining matrices across the panel.
Then
\[
 v_{\mathrm{piv}}
 =
 -\sum_{i=0}^{q_\rho-1}U_iv_i,
\]
and
\begin{equation}\label{eq:local-R-type}
R_s=
\begin{pmatrix}
	I_r & 0 & \cdots & 0 & -U_0\\
	0 & I_r & \cdots & 0  & -U_1\\
	\vdots & \vdots & \ddots & \vdots & \vdots \\
	0 & 0 & \cdots & I_r & -U_{q_{\rho}}
\end{pmatrix}.
\end{equation}
Every entry of the local reduction matrix is integral.

\subsection{The global reduction matrix for one segment}

The disappearance of the shuffle index allows the global reduction matrix to
be described by a single acyclic path model.

\begin{definition}[One-segment reduction graph]
\label{def:one-segment-reduction-graph}
Fix an affine--Bruhat shelling and the deterministic pivot rule of
Definition~\ref{def:deterministic-pivot-rule}.  The
\emph{one-segment reduction graph}
\(\Gamma_\Delta\) has the spanning indices as vertices and has a directed
edge \(\xi\to\xi'\) whenever solving a pivot relation for \(q_\xi\) produces
a term \(q_{\xi'}\) of strictly smaller shelling complexity.

Every edge is of one of the following two types:
\begin{enumerate}
 \item a \emph{panel edge}, coming from the sign relation
 \eqref{eq:panel-sign-relation};
 \item a \emph{Bernstein edge}, coming from affine-Hecke straightening and
 carrying a Bernstein translation monomial
 \(X^{\eta(e)}=\Theta_{\eta(e)}\) in the sense of
 Definition~\ref{def:Bernstein-translation-monomial}.
\end{enumerate}
All powers of \(q_\rho\) and \(q_\rho-1\), together with the type
intertwiner, are included in the integral edge operator.
\end{definition}

\begin{definition}[One-segment path data]
\label{def:one-segment-path-data}
Let
\[
 \gamma:
 \xi=\xi_0\longrightarrow\xi_1\longrightarrow\cdots
 \longrightarrow\xi_s=\mu
\]
be a directed path in \(\Gamma_\Delta\).  For an edge
\(e_k:\xi_{k-1}\to\xi_k\), let
\(\eta(e_k)\in\mathbb Z^\ell\) be the exponent of its Bernstein monomial and
let \(U(e_k)\) be the remaining integral scalar or type matrix after removing
the sign and pure translation character.  Define
\[
 \eta(\gamma)=\sum_{k=1}^s\eta(e_k),
 \qquad
 U(\gamma)=U(e_s)\cdots U(e_1).
\]
Let \(f(\gamma)\) be the parity-counted number of sign factors \(-1\) along
the path.  In the scalar-type case \(U(e_k)\in\Ocoeff\); in type dimension
\(r\),
\[
 U(\gamma)\in M_r(\Ocoeff).
\]
\end{definition}

\begin{lemma}[Acyclicity of the one-segment reduction]
\label{lem:one-segment-reduction-acyclic}
The graph \(\Gamma_\Delta\) is acyclic and every directed path terminates at
a nonpivot normal-form index.
\end{lemma}

\begin{proof}
Every directed edge strictly decreases the shell/order complexity, exactly
as in Lemma~\ref{lem:global-reduction-graph-acyclic}.
\end{proof}

\begin{proposition}[Exact one-segment path formula]
\label{prop:one-segment-exact-path-formula}
For every spanning index \(\xi\) and every normal-form index \(\mu\),
\begin{equation}\label{eq:R-one-segment-exact}
 \boxed{
 R_{\mu;\xi}
 =
 \sum_{\gamma:\xi\leadsto\mu}
 (-1)^{f(\gamma)}
 \chi_\Delta(X^{\eta(\gamma)})
 U(\gamma).
 }
\end{equation}
The sum is finite.  Using
\eqref{eq:segment-translation-character-general},
\begin{equation}\label{eq:R-one-segment-expanded}
 R_{\mu;\xi}
 =
 \sum_{\gamma:\xi\leadsto\mu}
 (-1)^{f(\gamma)}
 z_\Delta^{|\eta(\gamma)|}
 q_\rho^{-\sum_i(i-1)\eta_i(\gamma)}
 U(\gamma).
\end{equation}
\end{proposition}

\begin{proof}
Lemma~\ref{lem:one-segment-reduction-acyclic} gives finite termination.
Iterating the local panel matrix \eqref{eq:local-R-type} and the Bernstein
straightening relations gives one term for each directed path.  Bernstein
characters multiply, so the translation exponents add; the remaining type
operators multiply in path order.  Summing over all paths terminating at
\(\mu\) gives \eqref{eq:R-one-segment-exact}, and
\eqref{eq:segment-translation-character-general} gives the expanded form.
\end{proof}

\begin{corollary}[One-segment valuation estimate]
\label{cor:one-segment-path-valuation}
For every \(\mu,\xi\),
\begin{equation}\label{eq:R-one-segment-valuation}
 v_E(R_{\mu;\xi})
 \geq
 \min_{\gamma:\xi\leadsto\mu}
 \left[
 |\eta(\gamma)|v_E(z_\Delta)
 -
 \sum_i(i-1)\eta_i(\gamma)v_E(q_\rho)
 +
 v_E(U(\gamma))
 \right].
\end{equation}
If the minimum is represented by a unique path, or if the residues of the
minimal-valuation terms do not cancel, equality holds.
\end{corollary}

\begin{proof}
Apply the ultrametric inequality to
\eqref{eq:R-one-segment-expanded}.  Cancellation between distinct paths can
only increase the valuation.
\end{proof}

\subsection{Exact individual Cartan coefficients}

There is no shuffle index.  Let
\[
 e_\nu=q_{x,a}
\]
be a normal-form basis vector.  If
\[
 xa_j^{-1}=j_{x,j}x_j,
\]
then
\begin{equation}\label{eq:one-segment-smooth-A}
 \boxed{
 A^{\mathrm{sm}}_{\mu\nu}(a_j)
 =
 \sum_b
 R_{\mu;(x_j,b)}
 \lambda(j_{x,j}^{-1})_{ba}.
 }
\end{equation}
After tensoring with a torus weight vector \(v_\beta\),
\begin{equation}\label{eq:one-segment-locally-algebraic-A}
 \boxed{
 A_{(\mu,\beta),(\nu,\beta)}(a_j)
 =
 \eta_\beta(a_j)
 \sum_b
 R_{\mu;(x_j,b)}
 \lambda(j_{x,j}^{-1})_{ba}.
 }
\end{equation}
This is the exact matrix coefficient of the individual group element
\(a_j\).

\subsection{Chain of defects}

The Weil--Deligne module of a single segment has the chain
\[
 D_0,\ D_0(1),\ldots,D_0(\ell-1),
\]
with \(\operatorname{rk}D_0=d\).  Put
\[
 D^{(k)}
 =
 D_0\oplus\cdots\oplus D_0(k-1),
 \qquad
 1\leq k\leq\ell.
\]
The proper defects are
\begin{equation}\label{eq:one-segment-defects}
 \Delta_k
 =
 t_N(D^{(k)})
 -
 [E:F]
 \sum_{r=1}^{kd}\sum_\sigma i_{r,\sigma},
 \qquad
 1\leq k<\ell,
\end{equation}
and admissibility gives
\[
 \Delta_k\geq0,
 \qquad
 \Delta_\ell=0.
\]

If a normalized positive translation has noncentral coweight
\[
 \eta_{\mathrm{nc}}
 =
 \sum_{k=1}^{\ell-1}
 m_k\omega_{kd}^\vee,
 \qquad
 m_k\geq0,
\]
then the combined smooth and algebraic valuation is
\begin{equation}\label{eq:one-segment-translation}
 \boxed{
 v_E\left(
 \chi_\Delta(\Theta_\eta)
 \eta_{\mathrm{alg}}(\eta)
 \right)
 =
 \kappa_{E,F}
 \sum_{k=1}^{\ell-1}m_k\Delta_k.
 }
\end{equation}

\begin{definition}[Positively oriented one-segment path]
\label{def:positively-oriented-one-segment-path}
Let \(\gamma\in\mathcal G_j(\mu,\nu)\) be a reduction path contributing to
the coefficient of a positive fundamental Cartan generator \(a_j\).  We call
\(\gamma\) \emph{positively oriented} if its accumulated noncentral
translation has the form
\begin{equation}\label{eq:positive-one-segment-path-translation}
 \eta(\gamma)_{\mathrm{nc}}
 =
 \sum_{k=1}^{\ell-1}
 m_{\gamma,k}\omega_{kd}^\vee,
 \qquad
 m_{\gamma,k}\in\mathbb Z_{\geq0}.
\end{equation}
The central part is normalized using the total weak-admissibility equality,
so that its combined smooth--algebraic contribution has valuation zero.
\end{definition}

\begin{lemma}[Pathwise positivity]
\label{lem:pathwise-positivity-one-segment}
Assume
\[
 \Delta_k\geq0
 \qquad(1\leq k<\ell),
\]
and let \(\gamma\) be positively oriented in the sense of
Definition~\ref{def:positively-oriented-one-segment-path}.  Suppose that the
affine-Hecke scalar attached to \(\gamma\) is
\[
 \prod_\rho q_\rho^{c_\rho(\gamma)}
 (q_\rho-1)^{f_\rho(\gamma)},
 \qquad
 c_\rho(\gamma),f_\rho(\gamma)\geq0,
\]
and that \(U(\gamma)\in M_r(\Ocoeff)\).  Then the complete normalized
smooth--algebraic--type contribution of \(\gamma\) has nonnegative
valuation.  More precisely,

Define the \emph{complete normalized smooth--algebraic--type
contribution} of \(\gamma\) by
\begin{equation}\label{eq:one-segment-path-contribution-definition}
 \operatorname{Contr}(\gamma)
 :=
 \epsilon(\gamma)
 \left(
 \prod_\rho
 q_\rho^{c_\rho(\gamma)}
 (q_\rho-1)^{f_\rho(\gamma)}
 \right)
 \chi_\Delta\!\left(\Theta_{\eta(\gamma)}\right)
 \eta_{\mathrm{alg}}\!\left(\eta(\gamma)\right)
 U(\gamma)
 \in M_r(E),
\end{equation}
where \(\epsilon(\gamma)\in\{\pm1\}\) is the sign of the affine-Hecke
straightening coefficient.  Thus \(\operatorname{Contr}(\gamma)\) is the
matrix-valued contribution of the branch \(\gamma\) after combining the
smooth Bernstein factor, the algebraic torus factor, the affine-Hecke scalar,
and the type-intertwining factor.  When a scalar coefficient
\(A_{\mu\nu}(a_j)\) is formed, the contribution of \(\gamma\) is the
corresponding matrix entry of \(\operatorname{Contr}(\gamma)\); hence its
valuation is bounded below by the matrix valuation
\(v_E(\operatorname{Contr}(\gamma))\).

\begin{equation}\label{eq:pathwise-positive-valuation}
 v_E(\operatorname{Contr}(\gamma))
 \geq
 \kappa_{E,F}
 \sum_{k=1}^{\ell-1}
 m_{\gamma,k}\Delta_k
 +
 \sum_\rho c_\rho(\gamma)v_E(q_\rho)
 +
 v_E(U(\gamma))
 \geq0.
\end{equation}
\end{lemma}

\begin{proof}
By multiplicativity of the Bernstein character and of the algebraic torus
character, the accumulated translation factor along \(\gamma\) is evaluated
once on \(\eta(\gamma)\).  The central component contributes valuation zero
by the total weak-admissibility equality and the chosen normalization.  The
noncentral component is
\eqref{eq:positive-one-segment-path-translation}, so
\eqref{eq:one-segment-translation} gives
\[
 v_E\!\left(
 \chi_\Delta(\Theta_{\eta(\gamma)})
 \eta_{\mathrm{alg}}(\eta(\gamma))
 \right)
 =
 \kappa_{E,F}
 \sum_{k=1}^{\ell-1}
 m_{\gamma,k}\Delta_k
 \geq0.
\]
Every sign has valuation zero, every \(q_\rho-1\) is a unit, every
\(q_\rho^{c_\rho(\gamma)}\) has nonnegative valuation, and integrality of the
type gives \(v_E(U(\gamma))\geq0\).  Adding these contributions proves
\eqref{eq:pathwise-positive-valuation}.
\end{proof}

Consequently, if every path contributing to \(A_{\mu\nu}(a_j)\) is
positively oriented, then
\begin{equation}\label{eq:one-segment-A-lower-bound}
 v_E(A_{\mu\nu}(a_j))
 \geq
 \min_{\gamma\in\mathcal G_j(\mu,\nu)}
 \left[
 \kappa_{E,F}
 \sum_km_{\gamma,k}\Delta_k
 +
 \sum_\rho c_\rho(\gamma)v_E(q_\rho)
 +
 v_E(U(\gamma))
 \right]
 \geq0.
\end{equation}
Cancellation between different paths can only increase the left-hand
valuation.

\begin{proposition}[Explicit one-segment criterion]
\label{prop:one-segment}
Let \(G=\GL_n(F)\), \(n=d\ell\), and let
\[
 a_j=\operatorname{diag}(
 \underbrace{\varpi_F,\ldots,\varpi_F}_{j},
 1,\ldots,1),
 \qquad 1\leq j<n.
\]
Let \(\alpha_0\) be the max norm attached to the unweighted normal-form basis,
that is \(x_\nu=0\).  Assume the following conditions.

\begin{enumerate}
 \item[\textup{(H1)}] The proper defects satisfy
 \begin{equation}\label{eq:one-segment-criterion-Hu}
  \Delta_k\geq0,
  \qquad 1\leq k<\ell.
 \end{equation}

 \item[\textup{(H2)}] For every \(j\), every pair of normal-form indices
 \(\mu,\nu\), and every path
 \(\gamma\in\mathcal G_j(\mu,\nu)\) contributing to the positive Cartan
 coefficient \(A_{\mu\nu}(a_j)\),
 \begin{equation}\label{eq:one-segment-criterion-positive-cone}
  \eta(\gamma)_{\mathrm{nc}}
  =
  \sum_{k=1}^{\ell-1}
  m_{\gamma,k}\omega_{kd}^\vee,
  \qquad
  m_{\gamma,k}\geq0.
 \end{equation}

 \item[\textup{(H3)}] Every affine-Hecke straightening coefficient occurring
 on such a path is a product
 \begin{equation}\label{eq:one-segment-criterion-Hecke-integral}
  \epsilon\,
  \prod_\rho q_\rho^{c_\rho}
  (q_\rho-1)^{f_\rho},
  \qquad
  \epsilon\in\{\pm1\},
  \quad
  c_\rho,f_\rho\in\mathbb Z_{\geq0},
 \end{equation}
 and every type-intertwining product satisfies
 \begin{equation}\label{eq:one-segment-criterion-type-integral}
  U(\gamma)\in M_r(\Ocoeff).
 \end{equation}

 \item[\textup{(H4)}] The norm \(\alpha_0\) is \(K=\GL_n(\mathcal O_F)\)-invariant:
 \begin{equation}\label{eq:one-segment-criterion-K}
  \alpha_0(kv)=\alpha_0(v)
  \qquad(k\in K).
 \end{equation}

 \item[\textup{(H5)}] The central character \(\omega_V\) of the locally
 algebraic representation is unitary:
 \begin{equation}\label{eq:one-segment-criterion-centre}
  |\omega_V(z)|=1
  \qquad(z\in Z(G)).
 \end{equation}
\end{enumerate}

Then
\begin{equation}\label{eq:one-segment-positive-nonexpanding}
 \|a_j\|_{\alpha_0}\leq1
 \qquad(1\leq j<n).
\end{equation}
Moreover, let
\[
 z=\varpi_F I_n
\]
and let \(w_j\in K\) be the permutation matrix interchanging the first
\(j\) coordinates with the last \(n-j\) coordinates in the evident block
sense.  Then
\begin{equation}\label{eq:one-segment-inverse-explicit}
 a_j^{-1}
 =
 z^{-1}w_j a_{n-j}w_j^{-1}.
\end{equation}
Hence
\begin{equation}\label{eq:one-segment-negative-nonexpanding}
 \|a_j^{-1}\|_{\alpha_0}\leq1
 \qquad(1\leq j<n).
\end{equation}
Therefore every \(a_j\) acts isometrically on \(\alpha_0\); together with
\textup{(H4)}--\textup{(H5)} and Cartan decomposition, this gives a bounded
\(G\)-orbit, and hence a \(G\)-invariant norm.
\end{proposition}

\begin{proof}
By Lemma~\ref{lem:pathwise-positivity-one-segment} and
\textup{(H1)}--\textup{(H3)}, every path contribution to every coefficient
of \(a_j\) has nonnegative valuation.  The ultrametric inequality therefore
gives
\[
 v_E(A_{\mu\nu}(a_j))\geq0
\]
for all \(\mu,\nu,j\).  Since \(x_\nu=0\), this is exactly
\(\|a_j\|_{\alpha_0}\leq1\), proving
\eqref{eq:one-segment-positive-nonexpanding}.  The same inequality holds for
all positive powers by submultiplicativity.

The identity \eqref{eq:one-segment-inverse-explicit} is immediate from the
diagonal matrices: after conjugating \(a_{n-j}\) by \(w_j\), one obtains
\(\operatorname{diag}(I_j,\varpi_F I_{n-j})\), and multiplication by
\(z^{-1}\) gives
\(\operatorname{diag}(\varpi_F^{-1}I_j,I_{n-j})=a_j^{-1}\).
By \textup{(H4)} the factor \(w_j\) is an isometry and by \textup{(H5)} the
central factor \(z^{-1}\) is an isometry.  Applying
\eqref{eq:one-segment-positive-nonexpanding} to \(a_{n-j}\) gives
\eqref{eq:one-segment-negative-nonexpanding}.  Thus both \(a_j\) and
\(a_j^{-1}\) are nonexpanding, so \(a_j\) is an isometry.  Cartan
decomposition now gives boundedness of the full \(G\)-orbit; alternatively,
in the zero-potential situation just proved, the norm is already invariant
under the subgroup generated by \(K\), the centre, and the \(a_j\).
The bounded-orbit theorem of Section~5 then yields a \(G\)-invariant norm.
\end{proof}

The substantive local problem is therefore to verify the stated positivity
property for the chosen affine-Bruhat normal form.  The explicit
Bernstein--Iwahori formulas of \cite{HainesPettet2002} provide a concrete
route to this verification.

\subsection{Affine-Weyl checks for the fundamental Cartan generators}

We now make the affine-Weyl part of hypothesis \textup{(H2)} in
Proposition~\ref{prop:one-segment} completely explicit.  The formulas of
Haines--Pettet \cite{HainesPettet2002} are particularly well adapted to this
problem because the fundamental coweights of type \(A\) are minuscule.

\subsubsection{The affine-Weyl data}

Let
\[
 \widetilde W_\ell
 =
 X_*(T_\ell)\rtimes S_\ell
 =
 \mathbb Z^\ell\rtimes S_\ell
\]
be the extended affine Weyl group of the affine-Hecke algebra
\(\mathcal H_\ell\).  Write \(e_1,\ldots,e_\ell\) for the standard basis of
\(\mathbb Z^\ell\), put
\[
 \alpha_i=e_i-e_{i+1},
 \qquad
 1\leq i<\ell,
\]
and define the fundamental coweights
\begin{equation}\label{eq:block-fundamental-coweights}
 \omega_k^\vee
 =
 e_1+\cdots+e_k,
 \qquad
 1\leq k<\ell.
\end{equation}
Modulo the central coweight \(e_1+\cdots+e_\ell\), these generate the dominant
coweight cone.

For the one-segment Bernstein block, the coweight
\(\omega_k^\vee\) corresponds to the block Cartan element
\[
 b_k:=a_{kd},
 \qquad
 1\leq k<\ell.
\]
When \(d=1\), these are exactly all the fundamental Cartan generators of
\(\GL_\ell(F)\).  When \(d>1\), Haines--Pettet directly controls the
block-Cartan directions \(a_d,a_{2d},\ldots,a_{(\ell-1)d}\); the remaining
fundamental directions \(a_j\) with \(d\nmid j\) require, in addition, an
integral compact-mod-centre model for the internal \(\GL_d(F)\)-type.  This
extra internal-direction control is a separate requirement and does not
follow from the type \(A_{\ell-1}\) affine-Weyl combinatorics alone.

For \(1\leq k<\ell\),
\[
 \ell(t_{\omega_k^\vee})=k(\ell-k).
\]
A minimal alcove gallery from the base alcove to its
\(t_{\omega_k^\vee}\)-translate crosses precisely the affine walls attached
to the roots
\begin{equation}\label{eq:fundamental-translation-crossing-roots}
 \Phi_k
 =
 \{\,e_a-e_b:1\leq a\leq k<b\leq\ell\,\},
\end{equation}
each exactly once.  Thus every reduced expression has
\(N_k=k(\ell-k)\) simple affine reflections.

One convenient finite Weyl word is the Grassmannian word
\begin{equation}\label{eq:Grassmannian-word-fundamental}
 w_k
 =
 \prod_{r=k}^{1}
 \bigl(s_r s_{r+1}\cdots s_{r+\ell-k-1}\bigr),
\end{equation}
where the factors are written from left to right with \(r=k,k-1,\ldots,1\).
Its length is \(k(\ell-k)\), and its inversion set is exactly
\(\Phi_k\).  If \(\tau_k:=t_{\omega_k^\vee}w_k^{-1}\), then
\(\tau_k\) has length zero and
\[
 t_{\omega_k^\vee}=\tau_k w_k.
\]
Changing the convention for the affine rotation merely conjugates this
display and does not change the finite inequalities below.

\begin{definition}[Haines--Pettet finite support set]
\label{def:Haines-Pettet-finite-support}
Fix the positive/antidominant convention for Bernstein elements once and for
all.  For a fundamental coweight \(\omega_k^\vee\), let
\(\mathcal A_k\subset\widetilde W_\ell\) be the finite set of
Iwahori--Matsumoto basis indices which occur in the Haines--Pettet expansion
of \(\Theta^-_{\omega_k^\vee}\):
\begin{equation}\label{eq:HP-fundamental-expansion}
 \Theta^-_{\omega_k^\vee}
 =
 \sum_{x\in\mathcal A_k}
 c_{k,x}(q_\rho)\,T_x.
\end{equation}
Equivalently, \(\mathcal A_k\) is obtained by expanding a Haines--Pettet
minimal expression attached to a reduced word for
\(t_{\omega_k^\vee}\).  It is finite, with at most
\(2^{k(\ell-k)}\) subexpression branches before equal terms are collected.
\end{definition}

The minuscule formulas of Haines--Pettet give two facts which are the only
ones needed for valuations here.

\begin{lemma}[Haines--Pettet input for a fundamental translation]
\label{lem:HP-minuscule-input}
For each \(1\leq k<\ell\):
\begin{enumerate}
 \item the support \(\mathcal A_k\) is finite and is computable from any
 minimal expression of \(t_{\omega_k^\vee}\);
 \item after the unnormalized affine-Hecke convention
 \((T_s+1)(T_s-q_\rho)=0\) is used consistently, every coefficient in the
 minimal-expression expansion is an integral combination of products of
 \(q_\rho\) and \(q_\rho-1\).  In particular,
 \begin{equation}\label{eq:HP-coefficient-integrality}
  v_E(c_{k,x}(q_\rho))\geq0
 \end{equation}
 after the common normalization factor has been incorporated into the
 normalized smooth--algebraic translation character;
 \item the Iwahori--Matsumoto terms are organized by the translation part of
 the affine-Weyl index.  For the individual Bernstein element
 \(\Theta^-_{\omega_k^\vee}\), the relevant translation part is fixed by
 \(\omega_k^\vee\); the remaining variation is finite Weyl/Bruhat data.
\end{enumerate}
\end{lemma}

\begin{proof}
These are precisely the consequences of the minimal-expression and
minuscule formulas in \cite{HainesPettet2002}.  The reduced word has finite
length \(k(\ell-k)\), so expansion of the inverse/simple factors is finite.
The coefficients are obtained from the quadratic relation and therefore
from integral polynomials in the Hecke parameter.  The minuscule formula
groups the Iwahori--Matsumoto terms according to the fixed translation part
of the chosen Bernstein element.
\end{proof}

\subsubsection{The exact finite inequalities}

The preceding lemma isolates the affine-Weyl combinatorics from the
type--Bruhat normal-form problem.  For
\(x\in\mathcal A_k\), let \(U_{k,x}\) be the product of the type
intertwiners occurring when the \(T_x\)-term is inserted into the chosen
global reduction scheme.  Let
\[
 \beta_{k,x}
 :=
 v_E(c_{k,x}(q_\rho))
 +
 v_E(U_{k,x}).
\]
Under integral type normalization,
\begin{equation}\label{eq:HP-residual-nonnegative}
 \beta_{k,x}\geq0.
\end{equation}

For the block Cartan generator \(b_k=a_{kd}\), the normalized
smooth--algebraic valuation of the fixed translation part is
\begin{equation}\label{eq:HP-fixed-translation-value}
 \kappa_{E,F}\Delta_k.
\end{equation}
Thus the finite affine-Weyl inequalities which must be checked are
\begin{equation}\label{eq:HP-finite-termwise-inequalities}
 \boxed{
 \kappa_{E,F}\Delta_k+\beta_{k,x}\geq0
 \qquad
 (1\leq k<\ell,\ x\in\mathcal A_k).
 }
\end{equation}
The relation \eqref{eq:HP-finite-termwise-inequalities} follows immediately from
\[
 \Delta_k\geq0,
 \qquad
 \beta_{k,x}\geq0.
\]

There is, however, one further compatibility relation needed to transfer
this finite Hecke calculation to the exact individual-group-element matrix
coefficients used in this paper.

\begin{definition}[Haines--Pettet/type--Bruhat compatibility]
\label{def:HP-type-Bruhat-compatibility}
For the generator \(b_k=a_{kd}\), we say that the chosen deterministic
type--Bruhat reduction is \emph{Haines--Pettet compatible} if every global
reduction path contributing to an individual coefficient of \(b_k\) can be
partitioned into finite subfamilies indexed by
\(x\in\mathcal A_k\) such that:
\begin{enumerate}
 \item the accumulated Bernstein translation of every path in the
 \(x\)-subfamily has the same noncentral translation class as
 \(\omega_k^\vee\);
 \item all additional panel and type factors are integral;
 \item collecting the paths in the \(x\)-subfamily gives the coefficient
 \(c_{k,x}(q_\rho)U_{k,x}\), up to a unit.
\end{enumerate}
\end{definition}

Under this compatibility condition, the Haines--Pettet formula proves the
desired positivity without requiring every intermediate rank-one exponent
to remain in the dominant prefix cone.  This is stronger than the earlier
``positive-prefix at every intermediate step'' criterion, which is a useful
sufficient condition but is not necessary.

\begin{proposition}[Finite affine-Weyl criterion for the block Cartan directions]
\label{prop:HP-finite-affine-Weyl-criterion}
Assume
\[
 \Delta_k\geq0
 \qquad(1\leq k<\ell),
\]
the type intertwiners are integral, and the deterministic type--Bruhat
reduction is Haines--Pettet compatible for every
\(b_k=a_{kd}\).  Then
\[
 v_E(A_{\mu\nu}(b_k))\geq0
\]
for every \(\mu,\nu\) and every \(1\leq k<\ell\).  Equivalently,
the zero potential satisfies all positive block-Cartan inequalities.
\end{proposition}

\begin{proof}
Group the exact reduction paths according to
Definition~\ref{def:HP-type-Bruhat-compatibility}.  For a fixed
\(x\in\mathcal A_k\), the normalized translation valuation is
\(\kappa_{E,F}\Delta_k\), while the residual coefficient contributes
\(\beta_{k,x}\geq0\).  Thus each collected \(x\)-contribution has valuation
at least
\[
 \kappa_{E,F}\Delta_k+\beta_{k,x}\geq0.
\]
The ultrametric inequality then gives the same lower bound \(0\) for the sum
over \(x\).
\end{proof}

\subsubsection{Rank-one cross relations and the local check}

For comparison with the path language, the Bernstein cross relation
\eqref{eq:Bernstein-cross-relation-root-form} may be written, when
\(m=\langle\eta,\alpha\rangle\geq0\), as
\begin{equation}\label{eq:rank-one-Bernstein-string}
 T_s\Theta_\eta
 =
 \Theta_{s\eta}T_s
 +(q_\rho-1)
 \sum_{r=0}^{m-1}\Theta_{\eta-r\alpha^\vee},
\end{equation}
up to the orientation convention for \(\Theta_{-\alpha^\vee}\).
For a minuscule fundamental coweight every crossed root satisfies
\[
 \langle\omega_k^\vee,\alpha\rangle\in\{0,1\},
\]
so every root string has length at most one.  Applying the segment sign
character gives
\begin{equation}\label{eq:rank-one-after-sign-character}
 \chi_\Delta(T_s\Theta_\eta)
 =
 -\chi_\Delta(\Theta_{s\eta})
 +(q_\rho-1)
 \sum_{r=0}^{m-1}
 \chi_\Delta(\Theta_{\eta-r\alpha^\vee}).
\end{equation}
The finite-Hecke scalars therefore have valuation zero, and any additional
nonnegative power of \(q_\rho\) has nonnegative valuation.  This proves the
local integrality statement at one wall.  Haines--Pettet compatibility is
the statement that, after all these local relations are assembled along a
minimal expression and collected in the Iwahori--Matsumoto basis, the global
type--Bruhat reduction has exactly the finite support and fixed translation
part described above.

For reference, the earlier stronger sufficient condition can be stated as
follows.

\begin{definition}[Explicit positive-prefix hypothesis]
\label{def:explicit-positive-prefix-hypothesis}
At every rank-one straightening step
\[
 T_{s_j}\Theta_{\eta_j}
 =
 \Theta_{s_j\eta_j}T_{s_j}
 +(q_\rho-1)
 \sum_{r=0}^{\langle\eta_j,\alpha_j\rangle-1}
 \Theta_{\eta_j-r\alpha_j^\vee},
\]
require
\begin{align}
 &\langle\eta_j,\alpha_j\rangle\geq0,
 \label{eq:positive-prefix-pairing}\\
 &(s_j\eta_j)_{\mathrm{nc}}
 =
 \sum_{k=1}^{\ell-1}
 m_{j,\mathrm{pr},k}\omega_{kd}^\vee,
 \qquad
 m_{j,\mathrm{pr},k}\geq0,
 \label{eq:positive-prefix-principal}\\
 &(\eta_j-r\alpha_j^\vee)_{\mathrm{nc}}
 =
 \sum_{k=1}^{\ell-1}
 m_{j,r,k}\omega_{kd}^\vee,
 \qquad
 m_{j,r,k}\geq0.
 \label{eq:positive-prefix-corrections}
\end{align}
The same conditions are required recursively on every correction branch.
\end{definition}

\begin{proposition}[Explicit one-wall Bernstein--Iwahori positivity]
\label{prop:one-wall-Bernstein-Iwahori-positivity}
Put
\[
 h_j:=\langle\eta_j,\alpha_j\rangle\in\mathbb Z_{\geq0},
\]
so that the rank-one Bernstein relation at the \(j\)-th wall is
\begin{equation}\label{eq:explicit-one-wall-Bernstein-relation}
 T_{s_j}\Theta_{\eta_j}
 =
 \Theta_{s_j\eta_j}T_{s_j}
 +(q_\rho-1)
 \sum_{r=0}^{h_j-1}
 \Theta_{\eta_j-r\alpha_j^\vee}.
\end{equation}
For the principal branch set
\[
 \zeta_{j,\mathrm{pr}}:=s_j\eta_j,
\]
and for the \(r\)-th correction branch, \(0\leq r<h_j\), set
\[
 \zeta_{j,r}:=\eta_j-r\alpha_j^\vee.
\]
Assume the positive-prefix decompositions
\begin{align}
 (\zeta_{j,\mathrm{pr}})_{\mathrm{nc}}
 &=
 \sum_{k=1}^{\ell-1}
 m_{j,\mathrm{pr},k}\omega_{kd}^\vee,
 &
 m_{j,\mathrm{pr},k}&\in\mathbb Z_{\geq0},
 \label{eq:one-wall-principal-prefix-explicit}\\
 (\zeta_{j,r})_{\mathrm{nc}}
 &=
 \sum_{k=1}^{\ell-1}
 m_{j,r,k}\omega_{kd}^\vee,
 &
 m_{j,r,k}&\in\mathbb Z_{\geq0},
 \qquad 0\leq r<h_j,
 \label{eq:one-wall-correction-prefix-explicit}
\end{align}
and suppose
\begin{equation}\label{eq:one-wall-defects-nonnegative}
 \Delta_k\geq0
 \qquad(1\leq k<\ell).
\end{equation}

Fix the \(\Ocoeff\)-lattices in all finite-dimensional type spaces that
underlie the chosen type--Bruhat normal form.  Let
\[
 U_{j,\mathrm{pr}}
 \quad\text{and}\quad
 U_{j,r}\ \ (0\leq r<h_j)
\]
be the type-intertwining matrices attached respectively to the principal
branch and to the correction branches of
\eqref{eq:explicit-one-wall-Bernstein-relation}.  The hypothesis that the
type intertwiners are \emph{integral} means explicitly that, for the source
and target type lattices \(\Lambda_{\mathrm{in}}\) and
\(\Lambda_{\mathrm{out}}\) of each branch,
\begin{equation}\label{eq:one-wall-type-intertwiner-lattice-integrality}
 U(\Lambda_{\mathrm{in}})
 \subseteq
 \Lambda_{\mathrm{out}}.
\end{equation}
Equivalently, in the fixed integral bases,
\begin{equation}\label{eq:one-wall-type-intertwiner-matrix-integrality}
 U_{j,\mathrm{pr}}\in M_{r_{\mathrm{pr}}}(\Ocoeff),
 \qquad
 U_{j,r}\in M_{r_r}(\Ocoeff),
\end{equation}
with the evident rectangular variant when the source and target type blocks
have different ranks.  In terms of the matrix valuation
\eqref{eq:matrix-valuation-definition}, this is the pair of inequalities
\begin{equation}\label{eq:one-wall-type-intertwiner-valuation-integrality}
 v_E(U_{j,\mathrm{pr}})\geq0,
 \qquad
 v_E(U_{j,r})\geq0
 \quad(0\leq r<h_j).
\end{equation}

After the segment sign character is applied, write the complete normalized
coefficient matrix of the principal branch in the form
\begin{equation}\label{eq:one-wall-principal-complete-coefficient}
 M_{j,\mathrm{pr}}
 =
 u_{j,\mathrm{pr}}\,
 q_\rho^{\,c_{j,\mathrm{pr}}}
 (-1)\,
 \chi_{\mathfrak m}
 \!\left(\Theta_{\zeta_{j,\mathrm{pr}}}\right)
 \eta_{\mathrm{alg}}
 \!\left(\zeta_{j,\mathrm{pr}}\right)
 U_{j,\mathrm{pr}},
\end{equation}
and the complete normalized coefficient matrix of the \(r\)-th correction
branch in the form
\begin{equation}\label{eq:one-wall-correction-complete-coefficient}
 M_{j,r}
 =
 u_{j,r}\,
 q_\rho^{\,c_{j,r}}
 (q_\rho-1)\,
 \chi_{\mathfrak m}
 \!\left(\Theta_{\zeta_{j,r}}\right)
 \eta_{\mathrm{alg}}
 \!\left(\zeta_{j,r}\right)
 U_{j,r}.
\end{equation}
Here
\[
 u_{j,\mathrm{pr}},u_{j,r}\in\Ocoeff^\times,
 \qquad
 c_{j,\mathrm{pr}},c_{j,r}\in\mathbb Z_{\geq0},
\]
allow for the unit factors and nonnegative powers of the Hecke parameter
coming from the chosen local reduction convention.  Let $v_E(M)=\min_{a,b}v_E(M_{ab})$, where $M$ is a matrix. Then the principal branch satisfies the explicit identity
\begin{align}
 v_E(M_{j,\mathrm{pr}})
 &=
 c_{j,\mathrm{pr}}v_E(q_\rho)
 +
 \kappa_{E,F}
 \sum_{k=1}^{\ell-1}
 m_{j,\mathrm{pr},k}\Delta_k
 +
 v_E(U_{j,\mathrm{pr}})
 \notag\\
 &\geq0,
 \label{eq:one-wall-principal-valuation-nonnegative}
\end{align}
and, for every \(0\leq r<h_j\), the \(r\)-th correction branch satisfies
\begin{align}
 v_E(M_{j,r})
 &=
 c_{j,r}v_E(q_\rho)
 +
 v_E(q_\rho-1)
 +
 \kappa_{E,F}
 \sum_{k=1}^{\ell-1}
 m_{j,r,k}\Delta_k
 +
 v_E(U_{j,r})
 \notag\\
 &=
 c_{j,r}v_E(q_\rho)
 +
 \kappa_{E,F}
 \sum_{k=1}^{\ell-1}
 m_{j,r,k}\Delta_k
 +
 v_E(U_{j,r})
 \notag\\
 &\geq0.
 \label{eq:one-wall-correction-valuation-nonnegative}
\end{align}
In particular, every individual principal term and every individual
correction term in the rank-one expansion has nonnegative normalized
smooth--algebraic--type valuation.

The same conclusion holds recursively on every correction branch satisfying
the positive-prefix hypothesis.  More precisely, if a terminal term of a
recursive straightening is represented by a sequence of branch matrices
\[
 M_{\gamma}
 =
 M_{j_s,\varepsilon_s}\cdots
 M_{j_1,\varepsilon_1},
 \qquad
 \varepsilon_t\in
 \{\mathrm{pr}\}\cup
 \{0,\ldots,h_{j_t}-1\},
\]
then
\begin{equation}\label{eq:one-wall-recursive-product-valuation}
 v_E(M_\gamma)
 \geq
 \sum_{t=1}^s
 v_E(M_{j_t,\varepsilon_t})
 \geq0.
\end{equation}
Consequently every terminal branch coefficient is integral, and any finite
sum of such terminal coefficients is integral as well.
\end{proposition}

\begin{proof}
For every noncentral translation
\[
 \zeta_{\mathrm{nc}}
 =
 \sum_{k=1}^{\ell-1}m_k\omega_{kd}^\vee
 \qquad(m_k\geq0),
\]
the valuation identity established above gives
\begin{equation}\label{eq:one-wall-translation-valuation}
 v_E\!\left(
 \chi_{\mathfrak m}(\Theta_\zeta)
 \eta_{\mathrm{alg}}(\zeta)
 \right)
 =
 \kappa_{E,F}
 \sum_{k=1}^{\ell-1}m_k\Delta_k.
\end{equation}
The central part contributes valuation \(0\), because the normalized
smooth--algebraic central character is unitary by weak admissibility.
Applying \eqref{eq:one-wall-translation-valuation} to
\(\zeta_{j,\mathrm{pr}}\) and \(\zeta_{j,r}\) gives the terms appearing
in \eqref{eq:one-wall-principal-valuation-nonnegative} and
\eqref{eq:one-wall-correction-valuation-nonnegative}.

Next,
\[
 v_E(u_{j,\mathrm{pr}})
 =
 v_E(u_{j,r})
 =
 v_E(-1)
 =
 v_E(q_\rho-1)
 =
 0,
\]
because the \(u\)'s and \(q_\rho-1\) are units in \(\Ocoeff\).  Moreover,
\[
 c_{j,\mathrm{pr}}v_E(q_\rho)\geq0,
 \qquad
 c_{j,r}v_E(q_\rho)\geq0,
\]
since \(c_{j,\mathrm{pr}},c_{j,r}\geq0\), while
\eqref{eq:one-wall-type-intertwiner-valuation-integrality} gives
\[
 v_E(U_{j,\mathrm{pr}})\geq0,
 \qquad
 v_E(U_{j,r})\geq0.
\]
Finally, by
\eqref{eq:one-wall-principal-prefix-explicit},
\eqref{eq:one-wall-correction-prefix-explicit}, and
\eqref{eq:one-wall-defects-nonnegative},
\[
 \sum_{k=1}^{\ell-1}
 m_{j,\mathrm{pr},k}\Delta_k\geq0,
 \qquad
 \sum_{k=1}^{\ell-1}
 m_{j,r,k}\Delta_k\geq0.
\]
Adding these contributions proves
\eqref{eq:one-wall-principal-valuation-nonnegative} and
\eqref{eq:one-wall-correction-valuation-nonnegative}.

For a recursively produced terminal branch, the matrix valuation is
submultiplicative:
\[
 v_E(AB)\geq v_E(A)+v_E(B).
\]
Iterating this inequality gives
\eqref{eq:one-wall-recursive-product-valuation}.  If several terminal
branches contribute to the same matrix coefficient, the ultrametric
inequality gives
\[
 v_E\!\left(\sum_\gamma M_\gamma\right)
 \geq
 \min_\gamma v_E(M_\gamma)
 \geq0.
\]
Thus neither recursive multiplication nor collection of branches can
destroy integrality.
\end{proof}

\subsection{Example for \texorpdfstring{$n=3$}{n=3}: the length-three segment}

Take \(d=1\), \(\ell=3\), and
\[
 \Delta=[\chi,\chi\nu,\chi\nu^2].
\]
Then \(Q(\Delta)\) is the generalized Steinberg representation attached to
the length-three segment.  The affine Hecke algebra is of type \(A_2\), with
\[
 \chi_\Delta(T_1)=\chi_\Delta(T_2)=-1
\]
and
\[
 \chi_\Delta(X_1)=z_\Delta,
 \qquad
 \chi_\Delta(X_2)=z_\Delta q_\rho^{-1},
 \qquad
 \chi_\Delta(X_3)=z_\Delta q_\rho^{-2}.
\]

The two defects are
\begin{align}
 \Delta_1
 &=
 t_N(D_0)
 -
 [E:F]\sum_\sigma i_{1,\sigma},
 \label{eq:n3-one-segment-Delta1}\\
 \Delta_2
 &=
 t_N(D_0)+t_N(D_0(1))
 -
 [E:F]\sum_\sigma
 (i_{1,\sigma}+i_{2,\sigma}).
 \label{eq:n3-one-segment-Delta2}
\end{align}
Admissibility gives
\[
 \Delta_1,\Delta_2\geq0.
\]

For
\[
 a_1=\operatorname{diag}(\varpi_F,1,1),
 \qquad
 a_2=\operatorname{diag}(\varpi_F,\varpi_F,1),
\]
the normalized translation slopes are
\[
 \kappa_{E,F}\Delta_1,
 \qquad
 \kappa_{E,F}\Delta_2.
\]
At every \(s_1\)- or \(s_2\)-panel, the scalar local reduction matrix is
\[
 R_{s_i}=[I_{q_\rho}\mid-\mathbf 1].
\]
Thus a gallery with translation content
\[
 m_1\omega_1^\vee+m_2\omega_2^\vee
\]
contributes a term of normalized valuation at least
\begin{equation}\label{eq:n3-one-segment-gallery-valuation}
 \kappa_{E,F}
 (m_1\Delta_1+m_2\Delta_2).
\end{equation}

\bibliographystyle{alpha}
\addcontentsline{toc}{section}{References}
\bibliography{bib_norm}

\medskip
\noindent\textit{E-mail address}: a413xpyv@hotmail.com

\end{document}